\documentclass[a4paper, 11pt]{article}
\usepackage[left=25truemm,right=25truemm,top=30truemm,bottom=30truemm]{geometry}
\usepackage{amsmath}
\usepackage{amssymb}
\usepackage{bm}
\usepackage{mathrsfs}
\usepackage{graphicx}
\usepackage{array} 
\usepackage{float} 
\usepackage[all]{xy} 
\xyoption{poly}
\usepackage{ulem} 
\usepackage{color} 

\makeatletter
\@addtoreset{figure}{section}
\@addtoreset{table}{section}
\makeatother

\usepackage{amsthm}
\theoremstyle{plain}
\newtheorem{theorem}{Theorem}[section]
\newtheorem{corollary}[theorem]{Corollary} 
\newtheorem{lemma}[theorem]{Lemma} 
\newtheorem{proposition}[theorem]{Proposition} 
\newtheorem{question}[theorem]{Question} 
\theoremstyle{definition}
\newtheorem{definition}[theorem]{Definition}
\newtheorem{notation}[theorem]{Notation}
\theoremstyle{remark}
\newtheorem{remark}[theorem]{Remark} 
\newtheorem{example}[theorem]{Example}

\newcounter{num}
\newcommand{\Rnum}[1]{\setcounter{num}{#1} \Roman{num}}

\newcommand{\bC}{\mathbb{C}}
\newcommand{\bF}{\mathbb{F}}
\newcommand{\bQ}{\mathbb{Q}}
\newcommand{\bR}{\mathbb{R}}
\newcommand{\bT}{\mathbb{T}}
\newcommand{\bZ}{\mathbb{Z}}
\newcommand{\cB}{\mathcal{B}}
\newcommand{\cC}{\mathcal{C}}
\newcommand{\cO}{\mathcal{O}}
\newcommand{\cV}{\mathcal{V}}

\newcommand{\fraci}{\mathfrak{i}}
\newcommand{\fj}{\mathfrak{j}}

\newcommand{\rE}{\mathrm{E}}

\newcommand{\id}{\mathrm{id}}
\newcommand{\Id}{\mathrm{Id}}
\newcommand{\ob}{\mathrm{ob}}
\newcommand{\ch}{\operatorname{char}}
\newcommand{\hw}{\operatorname{hw}}
\newcommand{\idx}{\operatorname{idx}}
\newcommand{\Idx}{\operatorname{Idx}}
\newcommand{\Res}{\operatorname{Res}}
\newcommand{\Supp}{\operatorname{Supp}}

\title{Characteristic polynomials of isometries of\\ even unimodular lattices\footnote{MSC(2020): 11H56, 14J28.}}
\author{Yuta Takada\thanks{Department of Mathematics, Graduate School of Science, 
Hokkaido University, Kita 10, Nishi 8, Kita-ku, Sapporo 060-0810 Japan; 
JSPS Research Fellow. \texttt{takada@math.sci.hokudai.ac.jp}}}
\date{January 23, 2024}

\begin{document}
\maketitle

\begin{abstract}
E.~Bayer-Fluckiger gave a necessary and sufficient condition for a polynomial 
to be realized as the characteristic polynomial of a semisimple 
isometry of an even unimodular lattice, by describing the local-global obstruction,  
and the author extended the result. 
This article presents a systematic way to compute the obstruction. 
As an application, we give a necessary and sufficient condition for a Salem number of 
degree $10$ or $18$ to be realized as the dynamical degree of an automorphism of 
nonprojective K3 surface, in terms of its minimal polynomial. 
\end{abstract}

\section{Introduction}
In connection with the study of automorphisms of K3 surfaces,  
the following question was raised by B.H.~Gross and C.T.~McMullen \cite{GM02}. 

\begin{question}\label{q:1}
Which polynomial can be realized as the characteristic polynomial of 
an isometry of an even unimodular lattice with a prescribed signature?
\end{question}

In recent years, significant progress in areas concerning this question has been made  
by works of E.~Bayer-Fluckiger and L.~Taelman \cite{BT20}, and of 
Bayer-Fluckiger \cite{Ba20,Ba21,Ba22}. The author's work \cite{Ta23} follows them. 
This article gives subsequent results thereafter.

\subsection{Recent developments}
Let $r$ and $s$ be non-negative integers. 
It is known that $r \equiv s \bmod 8$ if $(r,s)$ is the signature of 
an even unimodular lattice. 
Let $F(X)\in \bZ[X]$ be a monic polynomial with $F(0)\neq 0$. 
We define a monic polynomial $F^*(X)\in \bQ[X]$ by 
$F^*(X) := F(0)^{-1}X^{\deg F} F(X^{-1})$, and 
say that $F$ is \textit{$*$-symmetric} if $F^* = F$. 
In this case, the constant term $F(0)$ is $1$ or $-1$, so we say that 
$F$ is \mbox{\textit{$+1$-symmetric}} or \textit{$-1$-symmetric} according to the value $F(0)$. 
Suppose that $F$ is the characteristic polynomial of a semisimple isometry of 
an even unimodular lattice of signature $(r,s)$. 
Then $F$ is a $*$-symmetric polynomial of even degree $(r + s)$. Moreover
\begin{equation}\label{eq:Sign_intro}
\text{$r,s \geq m(F)$ and if $F(1)F(-1)\neq 0$ then $r\equiv s\equiv m(F) \bmod 2$,} 
\tag{Sign}  
\end{equation}
where $m(F)$ is the number of roots of $F$ whose absolute values are greater than $1$
counted with multiplicity; and 
\begin{equation}\label{eq:Square_intro}
\text{$|F(1)|, |F(-1)|$ and $(-1)^{(\deg F)/2} F(1) F(-1)$ are all squares.} \tag{Square}
\end{equation}
Gross and McMullen speculated that these necessary conditions 
for $F$ to be realized as a characteristic polynomial are also sufficient 
when $F$ is irreducible, and they showed that if $F$ is irreducible and 
the assumption \eqref{eq:Square_intro} 
is replaced by the more stronger assumption $(-1)^{(\deg F)/2} F(1)F(-1) = 1$, 
then these are sufficient. Afterwards, 
Bayer-Fluckiger and Taelman \cite{BT20} proved that the speculation is 
correct by using a local-global theory. 

Bayer-Fluckiger \cite{Ba20,Ba21,Ba22} proceeded to the case where the polynomial $F$ 
is reducible and $+1$-symmetric. In this case, the conditions \eqref{eq:Sign_intro} and 
\eqref{eq:Square_intro} are not sufficient as pointed out in \cite{GM02}. 
She showed that the condition \eqref{eq:Square_intro} is necessary and sufficient for 
the existence of an even unimodular lattice over $\bZ_p$ having a semisimple isometry 
with characteristic polynomial $F$ for every prime $p$, where $\bZ_p$ is the 
ring of $p$-adic integers. On the other hand, 
the condition \eqref{eq:Sign_intro} is obtained by considering over $\bR$. 
Hence, we can say that these two conditions are local. 
For a $+1$-symmetric polynomial $F$ with \eqref{eq:Sign_intro} and 
\eqref{eq:Square_intro}, 
Bayer-Fluckiger gave a necessary and sufficient condition for $F$ to be 
realized as the characteristic polynomial of a semisimple isometry of an even unimodular 
lattice over $\bZ$ of signature $(r,s)$, by describing the local-global obstruction. 
These results were reformulated and extended to the case where $F$ is $*$-symmetric, 
which covers the $-1$-symmetric case, in the author's article \cite{Ta23}. 
This article presents a systematic way to compute the obstruction. 
In order to explain it, we describe the local-global obstruction in 
\S\ref{ss:obstruction} (a more detailed description will be given in \S\ref{sec:LGPO}).

\subsection{Obstruction}\label{ss:obstruction}
We begin by defining an invariant of an isometry of an inner product space over $\bR$, 
called the \textit{index}. Let $t$ be an isometry of an inner product space $V$ 
over $\bR$ of signature $(r,s)$, and let $F\in \bR[X]$ be its characteristic polynomial. 
Then $V$ decomposes as $V = \sum_f V(f;t)$, where 
$V(f;t) := \{ v\in V \mid f(t)^N.v = 0 \text{ for some $N\in \bZ_{\geq 0}$} \}$ 
and $f$ ranges over the irreducible factors of $F$ in $\bR[X]$. 
Let $I(F;\bR)$ denote the set of $*$-symmetric irreducible factors of $F$. 
The \textit{index} $\idx_t$ of $t$ is the map from $I(F;\bR)$ to $\bZ$ 
defined by $\idx_t(f) = r_f - s_f$, where $(r_f, s_f)$ is the signature of $V(f; t)$.

For a $*$-symmetric polynomial $F$ and non-negative integers $r,s$ with 
$r + s = \deg(F)$, we write $\Idx(r,s;F)$ for the set of maps $I(F;\bR) \to \bZ$ 
expressed as $\idx_t$ for some semisimple isometry $t$ of an inner product space $V$ of 
signature $(r,s)$, with characteristic polynomial $F$. 
Each map in $\Idx(r,s;F)$ is referred to as an \textit{index map} 
(see Definition \ref{def:indexmap} and Theorem \ref{th:Sign_implies}). 
Furthermore, we refer to an isometry with characteristic polynomial $F$ and with 
index $\fraci\in \Idx(r,s;F)$ as an \mbox{\textit{$(F, \fraci)$-isometry}} for short. 
As a detailed version of Question \ref{q:1}, we investigate when the following 
condition holds for a given index map $\fraci\in \Idx(r,s;F)$. 
\begin{equation}\label{eq:spade}
\begin{minipage}{14cm}
There exists an even unimodular lattice over $\bZ$ 
of signature $(r,s)$ having a semisimple $(F, \fraci)$-isometry.  
\end{minipage}
\end{equation}

We next define a $\bQ[X]$-module for a polynomial. 
Let $F\in \bZ[X]$ be a $*$-symmetric polynomial. 
For a factor $f$ of $F$, we write $m_f$ for the multiplicity of $f$ in $F$. 
Let $I(F;\bQ)$ denote the set of $*$-symmetric irreducible factors of $F$, 
and $I_2(F;\bQ)$ the set of non-$*$-symmetric irreducible factors of $F$ in $\bQ[X]$. 
Furthermore, we put $I_1(F;\bQ) = I(F;\bQ) \setminus\{X-1,X+1\}$ 
(see Definitions \ref{def:type012} and \ref{def:type_i_component} for 
the meaning of subscripts). Then $F$ can be expressed as 
\[ F(X) 
= (X-1)^{m_+} (X+1)^{m_-} 
\times \prod_{f\in I_1(F;\bQ)} f(X)^{m_f} 
\times \prod_{\{g,g^*\}\subset I_2(F;\bQ)} (g(X)g^*(X))^{m_g}, 
\] 
where $m_\pm := m_{X\mp1}$. 
For a factor $f$ which is in $I(F;\bQ)$ or of the form $gg^*$ for some $g\in I_2(F;\bQ)$, 
we define $M^f := (\bQ[X]/(f))^{m_f}$, and   
\begin{equation*}
 M := M^+ \times M^- \times \prod_{f\in I_1(F;\bQ)} M^f 
\times \prod_{\{g,g^*\}\subset I_2(F;\bQ)} M^{gg^*},  
\end{equation*}
where $M^\pm := M^{X\mp1}$. 
Let $\alpha:M\to M$ be the linear transformation defined by 
the multiplication by $X$. Then, it is a semisimple transformation with 
characteristic polynomial $F$. 
Furthermore the $\bQ$-algebra $M$ can be seen as a $\bQ[X]$-module 
by the action of $\alpha$. In this case, we refer to $M$ as the 
\textit{associated $\bQ[X]$-module of $F$} with transformation $\alpha$. 

A key idea in tackling Question \ref{q:1} is to consider when there exists an 
inner product $b:M\times M\to \bQ$ on the $\bQ$-vector space $M$ such that $\alpha$ 
becomes an isometry having a prescribed index $\fraci$ and the inner product space $(M, b)$ 
contains an $\alpha$-stable even unimodular lattice.  
By reinterpreting Question \ref{q:1} as an existence problem for inner products 
in this way, a local-global argument works well. 

Let $\cV$ denote the set of all places of $\bQ$. 
In the following, we fix the following data: 
non-negative integers $r$ and $s$ with $r\equiv s \bmod 8$; 
a $*$-symmetric polynomial $F\in \bZ[X]$ with the conditions \eqref{eq:Sign_intro} 
and \eqref{eq:Square_intro}; and an index map $\fraci\in \Idx(r,s;F)$.   
We write $I = I(F;\bQ)$, $I_1 = I_1(F;\bQ)$, and $I_2 = I_2(F;\bQ)$ for short. 
Let $M$ be the associated $\bQ[X]$-module with transformation $\alpha$. 
For each place $v\in \cV$, we define $M_v := M\otimes\bQ_v$. 
Similarly $M^f_v := M^f\otimes\bQ_v$ for $f$ which is in $I$ or 
of the form $gg^*$ for some $g\in I_2$. Then 
\[ M_v = M^+_v \oplus M^-_v 
\oplus \bigoplus_{f\in I_1} M^f_v 
\oplus \bigoplus_{\{g,g^*\} \subset I_2} M^{gg^*}_v  
\]
as $\bQ_v[X]$-modules, where $\alpha:M\to M$ is extended to a 
$\bQ_v$-linear transformation on $M_v$ in a unique way. 
For each $v\in \cV$, we consider the following three properties 
\eqref{eq:P1_intro}--\eqref{eq:P3_intro} of an inner product 
$b_v:M_v\times M_v\to \bQ_v$ on $M_v$. The first property is that
\begin{equation}\label{eq:P1_intro}
\text{$\alpha:M_v \to M_v$ is an isometry with respect to $b_v$.} \tag{P1}
\end{equation}
Assume that $b_v$ has the property \eqref{eq:P1_intro}. The second property is that
\begin{equation}\label{eq:P2_intro}
\begin{split}
&\text{if $v\neq \infty$ then there exists an $\alpha$-stable
even unimodular lattice over $\bZ_v$ on $(M_v, b_v)$, and}\\
&\text{if $v = \infty$ then the isometry $\alpha$ of $(M_\infty, b_\infty)$ 
has index $\fraci$.}
\end{split}\tag{P2}
\end{equation}
Let $b_v|_{M_v^\pm}$ denote the inner product on $M_v^\pm$ obtained by 
restricting $b_v$ to $M_v^\pm \times M_v^\pm \subset M_v \times M_v$. 
The last property is that
\begin{equation}\label{eq:P3_intro}
\det(b_v|_{M_v^\pm}) = \begin{cases}
(-1)^{(m_\pm - \fraci(X\mp1))/2}\, |F_{12}(\pm1)| &\text{if $m_+$ is even} \\
(-1)^{(m_\pm - \fraci(X\mp1))/2}\, 2|F_{12}(\pm1)| &\text{if $m_+$ is odd} 
\end{cases}
\quad\text{in $\bQ_v^\times/\bQ_v^{\times 2}$,}\tag{P3}
\end{equation}
where $\det$ is the determinant and 
$F_{12}(X) := F(X)/((X-1)^{m_+}(X+1)^{m_-}) \in \bZ[X]$. 
We write $\cB_\fraci$ for the set of families $\{b_v\}_{v\in \cV}$ of 
inner products on $M_v$ such that each $b_v$ has the properties 
\eqref{eq:P1_intro}--\eqref{eq:P3_intro} and 
$\#\{ v \in \cV \mid \hw_v(b_v|_{M^f_v}) \neq 0 \}$ is finite for all $f\in I$, 
where $\hw_v$ is the Hasse-Witt invariant (taking values in $\bZ/2\bZ$, 
see \S\ref{ss:IPSandL}). 
The conditions \eqref{eq:Sign_intro} and \eqref{eq:Square_intro}
guarantee that $\cB_\fraci$ is not empty. 

If $b$ is an inner product on $M$ such that $\alpha:M\to M$ becomes an isometry 
having index $\fraci$ and $(M,b)$ contains an $\alpha$-stable even unimodular lattice 
over $\bZ$, then the family $\{b\otimes \bQ_v\}_{v\in \cV}$ of inner products 
$b\otimes \bQ_v :M_v \times M_v \to \bQ_v$ obtained by localizations 
belongs to $\cB_\fraci$. 
The local-global principle for the existence of such an inner product on $M$ 
is described as the equivalence of the following two conditions
(Theorem \ref{th:LGP}):
\begin{equation}\label{eq:club}
\begin{minipage}{14cm}
There exists an inner product $b$ on $M$ such that $\alpha:M\to M$ becomes an isometry 
having index $\fraci$ and $(M, b)$ contains 
an $\alpha$-stable even unimodular lattice over $\bZ$.  
\end{minipage}
\end{equation}
\begin{equation}\label{eq:diamond}
\begin{minipage}{14cm}
There exists a family $\{b_v\}_{v\in \cV}\in \cB_\fraci$ 
such that $\sum_{v\in \cV} \hw_v(b_v|_{M^f_v}) = 0$ for any $f\in I$. 
\end{minipage}
\end{equation}

We remark that the former condition \eqref{eq:club} is equivalent to \eqref{eq:spade}. 
Let us rephrase the latter condition \eqref{eq:diamond} further. 
Let $C(I)$ denote the $\bZ/2\bZ$-module consisting of all maps from 
$I$ to $\bZ/2\bZ$, that is,  
$C(I) := \{ \gamma: I\to \bZ/2\bZ \} = (\bZ/2\bZ)^{\oplus I}$.  
Moreover, we define a map $\eta : \cB_\fraci \to C(I)$ by
\[ \eta(\{b_v\}_{v})(f) = \sum_{v\in \cV}\hw_v(b_v|_{M^f_v}) \in \bZ/2\bZ
\quad\textstyle (\{b_v\}_{v}\in \cB_\fraci, \,f\in I ).
\]
Under this notation, the condition \eqref{eq:diamond} can be rephrased as the one that 
there exists a family $\{b_v\}_{v} \in \cB_\fraci$ such that $\eta(\{b_v\}_{v}) = \bm{0}$, 
where $\bm{0} \in C(I)$ is the zero map. 
For a prime $p$ and a monic polynomial $f\in \bZ[X]$, we define 
\[ \overline{I(f;\bQ_p)} := 
\left\{ \bar{h} \in \bF_p[X] \;\middle|\; 
\begin{tabular}{l}
\text{$\bar{h}$ is irreducible, and there exists a $*$-symmetric } \\
\text{irreducible factor of $f$ in $\bZ_p[X]$ whose reduction } \\
\text{modulo $p$ is divisible by $\bar{h}$ in $\bF_p[X]$}
\end{tabular}
\right\}.
\]
Moreover, for two monic polynomials $f$ and $g\in \bZ[X]$, 
we define a set $\Pi(f,g)$ of primes by 
\[ \Pi(f,g) := \{ p:\text{prime} 
\mid \overline{I(f;\bQ_p)}\cap \overline{I(g;\bQ_p)} \neq \emptyset \}. \]
For simplicity of explanation, we assume that 
each of the multiplicities $m_+$ and $m_-$ of $X-1$ and $X+1$ is $0$ or at least $3$. 
Let $\sim$ be the equivalence relation on $I$ generated by the binary relation
$\{ (f,g) \in I\times I \mid \Pi(f,g)\neq \emptyset \}$, and let $\Omega$ be 
the submodule 
$\{c \in C(I) \mid \text{ $c(f) = c(g)$ if $f \sim g$ } \}$ 
of $C(I)$. We will see that the homomorphism 
\[ \textstyle\Omega \to \bZ/2\bZ, \, c \mapsto \sum_{f\in I}\eta(\{b_v\}_{v})(f) \cdot c(f)  \]
is defined independently of the choice of the family $\{b_v\}_{v}\in \cB_\fraci$.  
This homomorphism is called the \textit{obstruction map} for $(F, \fraci)$ and 
denoted by $\ob_\fraci : \Omega \to \bZ/2\bZ$. 
The submodule $\Omega \subset C(I)$ is called the \textit{obstruction group} 
for $(F, \fraci)$ 
(it does not depend on $\fraci$ under the current assumption on $m_+$ and $m_-$). 
It is obvious that if there exists a family $\{b_v\}_{v} \in \cB_\fraci$ such that 
$\eta(\{b_v\}_{v}) = \bm{0}$ then
\begin{equation}\label{eq:heart}
\text{the obstruction map $\ob_\fraci:\Omega \to \bZ/2\bZ$ is the zero map.}
\end{equation}

\begin{theorem}\label{th:A}
Let $r,s$ be non-negative integers with $r\equiv s \bmod 8$, 
$F\in \bZ[X]$ a $*$-symmetric polynomial of degree $r+s$ with the conditions 
\eqref{eq:Sign_intro} and \eqref{eq:Square_intro}, 
and $\fraci\in \Idx(r,s;F)$ an index map. 
Assume that each of $m_+$ and $m_-$ is $0$ or at least $3$. 
The conditions \eqref{eq:spade}, \eqref{eq:club}, \eqref{eq:diamond}, 
and \eqref{eq:heart} are all equivalent.   
\end{theorem}

The version without the assumption on $m_+$ and $m_-$ will be explained again in \S3.  
In the case $m_+\in \{1,2\}$ or $m_-\in \{1,2\}$, the definition of the set $\Pi(f,g)$ 
(and hence that of the equivalence relation $\sim$ on $I$) needs to be modified, 
and if $m_\pm = 2$ then it depends on the value $\fraci(X\mp1)$.  
The equivalence \eqref{eq:spade} $\Leftrightarrow$ \eqref{eq:heart} was first 
proved by Bayer-Fluckiger in \cite{Ba21}, under the assumption that $F$ is 
$+1$-symmetric. 
On the framework established by her, 
the author \cite{Ta23} modified the definition of obstruction, and moreover, 
extended the theorem to the case where $F$ is $*$-symmetric, 
which covers the $-1$-symmetric case, mainly by careful analysis at the prime $2$, 
see also Remark \ref{rem:history_for_LGO}.

\subsection{Main results}
Let $F\in \bZ[X]$ be a $*$-symmetric polynomial of even degree with 
the condition \eqref{eq:Square_intro}. As in \S\ref{ss:obstruction}, the symbols 
$M$ and $M_\infty$ denote the associated $\bQ[X]$-module with transformation $\alpha$ 
and its localization at the infinite place $\infty$ respectively. 
Let $r,s$ be non-negative integers with $r\equiv s \bmod 8$ such that 
$F$ satisfies the condition \eqref{eq:Sign_intro}${}_{r,s}$ (we write 
\eqref{eq:Sign_intro}${}_{r,s}$ instead of \eqref{eq:Sign_intro} if necessary), 
and $\fraci\in \Idx(r,s;F)$ an index map. 

Then, there exists an inner product $b_\fraci$ on $M_\infty$
which makes $\alpha$ an isometry with index $\fraci$. 
Let $\eta_\infty(\fraci)$ denote the map from $I:= I(F;\bQ)$ to $\bZ/2\bZ$ defined by 
\[\eta_\infty(\fraci)(f) = \hw_\infty(b_\fraci|_{M^f_\infty}) \quad (f\in I).  \]
This map is uniquely determined by $\fraci$ 
and can be expressed explicitly, see \S\ref{ss:comparison}. 

Let $r',s'$ be non-negative integers with $r'\equiv s' \bmod 8$ such that 
$F$ satisfies the condition \eqref{eq:Sign_intro}${}_{r',s'}$, and 
$\fj\in \Idx(r',s';F)$ an index map. We will see that if 
\begin{equation}\label{eq:idx_mod4_intro}
 \fraci(X-1) \equiv \fj(X-1) \quad\text{and}\quad 
\fraci(X-1) \equiv \fj(X-1) \mod 4
\end{equation}
then two equivalence relations on $I$ defined by $(F, \fraci)$ and $(F, \fj)$ are 
the same. In this case, the obstruction group $\Omega$ for $(F, \fraci)$ 
is the same as that for $(F, \fj)$. 

\begin{theorem}\label{th:comparison_intro}
Let $F\in \bZ[X]$ be a $*$-symmetric polynomial of even degree with 
the condition \eqref{eq:Square_intro}, 
and $r,s, r', s'$ non-negative integers with $r\equiv s$ 
and $r'\equiv s' \bmod 8$ such that \eqref{eq:Sign_intro}${}_{r,s}$ and 
\eqref{eq:Sign_intro}${}_{r',s'}$ hold for $F$. 
Let $\fraci\in \Idx(r,s;F)$ and $\fj\in \Idx(r',s';F)$ be index maps 
satisfying \eqref{eq:idx_mod4_intro}. Then we have 
\[ \ob_\fraci(c) = \ob_\fj(c) + (\eta_\infty(\fraci) - \eta_\infty(\fj))\cdot c \]
for all $c\in \Omega$. 
\end{theorem}

This theorem tell us that the difference of two obstruction maps can be 
calculated only by information of index maps. 

Let $m_\pm$ be the multiplicity of $X\mp1$ in $F$, 
and put $F_{12}(X) = F(X)/((X-1)^{m_+}(X+1)^{m_-}) \in \bZ[X]$. 
For any index map $\fraci$, if we put $i_+ = \fraci(X-1)$ and $i_- = \fraci(X+1)$ then 
\begin{equation}\label{eq:prolong_intro}
\begin{split}
&-m_+ \leq i_+ \leq m_+,\, -m_- \leq i_- \leq m_-,\, 
i_+ \equiv i_- \equiv m_+ \bmod 2, \quad\text{and}\quad \\
&i_+ + i_- \equiv 1 - e(F_{12}) \bmod 4,  
\end{split}
\end{equation}
where $e(F_{12}) \in \{1,-1\}$ is the signature of $F_{12}(1)F_{12}(-1)$, 
see Theorem \ref{th:Sign_implies} and Proposition \ref{prop:index_prolong}. 

\begin{theorem}\label{th:chpl_on_index0_intro}
Let $F\in \bZ[X]$ be a $*$-symmetric polynomial of even degree $2n$, and 
$i_+, i_-\in \bZ$ integers with \eqref{eq:prolong_intro}. Assume that 
\begin{itemize}
\item[\textup{(a)}] $m_+ \neq 1$ and $m_- \neq 1$, or 
\item[\textup{(b)}] $F$ is a product of cyclotomic polynomials. 
\end{itemize}
Then there exists $\fj \in \Idx(n,n;F)$ with 
$\fj(X-1) \equiv i_+$ and $\fj(X+1) \equiv i_-$ mod $4$ such that an even 
unimodular lattice of signature $(n,n)$ admits a semisimple $(F,\fj)$-isometry. 
\end{theorem}

The case (a) is a refined version of \cite[Theorem 1.2]{Ta23}. 
In the proof of Theorem \ref{th:chpl_on_index0_intro}, the index $\fj\in \Idx(n,n;F)$
is given constructively. So, the obstruction map $\ob_\fraci$ for any index map $\fraci$
can be calculated by using Theorem \ref{th:comparison_intro}, that is, 
by comparing it with $\ob_\fj$. 

As an application to the study of automorphisms of K3 surfaces, we get the 
following theorem, see \S\ref{sec:DDofK3SA} for terms therein. 

\begin{theorem}\label{th:NPR}
Let $\lambda$ be a Salem number of degree $d = 10$ or $18$, and 
$S$ its minimal polynomial. 
Let $\cC_{10}$ and $\cC_{18}$ be the sets consisting of integers defined by 
\[ \begin{split}
\cC_{10} &:= \{1,2,3,4,5,6,7,8,9,10,11,12,14,15,16,18,21,22,24,28,30,36,42\}, \\
\cC_{18} &:= \{1,2,3,4,6,12\}. 
\end{split}
\]
Then $\lambda$ is realizable as the dynamical degree of an automorphism of 
a nonprojective K3 surface if and only if 
there exists $l\in \cC_d$ such that $\Pi(S, \Phi_l)\neq \emptyset$. 
Here $\Phi_l$ is the $l$-th cyclotomic polynomial. 
\end{theorem}

In the case $d = 18$, a similar result is obtained by Bayer-Fluckiger, see 
\cite[Theorem 22.1]{Ba22}. She also gives an example of a Salem number of degree $18$ 
that is not realizable as the dynamical degree of an automorphism of a K3 surface, 
projective or not (\cite[Example 26.3]{Ba22}). 
In the case $d = 10$, a similar result of the if part is obtained by her, see 
\cite[Theorem 23.1]{Ba22}. 
Compared with her results, we remark that the assumption that $S$ is `unramified' 
is unnecessary in Theorem \ref{th:NPR}. 
A Salem number of degree $10$ is not nonprojectively realizable if and only if 
$\Pi(S,\Phi_l) = \emptyset$ for all $l\in \cC_{10}$ by Theorem \ref{th:NPR}, 
where $S$ is the corresponding Salem polynomial. 
This condition is so strong that there seems to be no number satisfying it, 
but the author has no proof.   
For the other degrees, more simpler criteria are already known, 
see Remark \ref{otherdegrees}. 

The organization of this article is as follows.
We introduce terminology related to inner products and isometries in \S \ref{sec:IPandI}, 
and explain the local-global theory for Question \ref{q:1} in \S \ref{sec:LGPO}. 
Section \ref{sec:Comp_of_Ob} gives some results for computing the local-global 
obstruction. In particular, Theorem \ref{th:comparison_intro} is proved in this 
section. The proof of Theorem \ref{th:chpl_on_index0_intro} is given in 
\S5. We show Theorem \ref{th:NPR} in \S6.

\paragraph{Acknowledgments}
This work is supported by JSPS KAKENHI Grant Number JP22KJ0009.

\section{Inner products and isometries}\label{sec:IPandI}
We here summarize terms and results related to inner products and isometries. 

\subsection{Inner product spaces and lattices}\label{ss:IPSandL}
Let $K$ be a field, and $V$ a finite dimensional $K$-vector space. 
An \textit{inner product} on $V$ is a nondegenerate symmetric bilinear form 
$b:V\times V \to K$. If $b$ is an inner product then the pair $(V, b)$ is 
called an \textit{inner product space}. 
Let $(V, b)$ be a $d$-dimensional inner product space over $K$, and 
let $e_1, \ldots, e_d$ be a basis of $V$. The $d\times d$ matrix 
$G := (b(e_i, e_j))_{ij}\in M_d(K)$ is called the \textit{Gram matrix} of $(V, b)$ 
with respect to the basis $e_1, \ldots, e_d$. 
The class of $\det G$ in $K^\times/K^{\times 2}$ is independent of the choice of 
the basis. This square class $\det G \in K^\times/K^{\times 2}$ is referred to as 
the \textit{determinant} of $(V,b)$ and denoted by $\det b$. 

For a place $v$ of $\bQ$, the symbol $\bQ_v$ denotes the field of $v$-adic numbers
if $v$ is a prime number, and the field $\bR$ if $v$ is the infinite place $\infty$.  
For an inner product $b$ over $\bQ_v$, we write $\hw_v(b)\in \{0,1\} = \bZ/2\bZ$ for 
its \textit{Hasse-Witt invariant}. 
This is the additive expression of the invariant $\epsilon$ in IV-\S2.1 of Serre's book 
\cite{Se73}. 

\medskip

Let $\cO$ be a principal ideal domain. 
A \textit{lattice} over $\cO$ is a pair $(\Lambda, b)$ consisting of a finitely generated 
free $\cO$-module $\Lambda$ and an inner product (i.e., 
a nondegenerate symmetric bilinear form) $b:\Lambda \times \Lambda \to \cO$. 
Let $(\Lambda, b)$ be a lattice of rank $d$ over $\cO$, and  
let $e_1, \ldots, e_d$ be a basis of $\Lambda$.
Similarly to the definitions for an inner product space, the $d\times d$ matrix 
$G := (b(e_i, e_j))_{ij}\in M_d(\cO)$ is called the Gram matrix of $(\Lambda, b)$, 
and the square class $\det b := \det G \in \cO^\times/ \cO^{\times 2}$ is 
called the determinant. We say that $(\Lambda, b)$ is \textit{unimodular} if its 
Gram matrix is invertible over $\cO$, and $(\Lambda, b)$ is \textit{even} is 
$b(x,x)\in 2\cO$ for all $x\in \Lambda$.

For a lattice over $\bZ$, its \textit{signature} is the signature of 
the inner product space over $\bR$ obtained by extending scalars to $\bR$. 
As for even unimodular lattices over $\bZ$, the following theorem is well known. 

\begin{theorem}\label{th:EULoverZ}
Let $r$ and $s\in \bZ_{\geq 0}$ be non-negative integers.  
\begin{enumerate}
\item If $(r,s)$ is the signature of an en even unimodular lattice over $\bZ$ 
then $r\equiv s \mod 8$. 
\item Suppose that $r\equiv s \mod 8$. Then there exists an even unimodular lattice 
over $\bZ$ of signature $(r,s)$. 
Moreover, if $r$ and $s$ are positive then such a lattice is unique up to isomorphism. 
\end{enumerate}
\end{theorem}
\begin{proof}
See \cite[Chapter V]{Se73}. 
\end{proof}

It is known that the \textit{$E_8$-lattice} $\rE_8$ is a unique even unimodular lattice 
of signature $(8, 0)$. However, in Theorem \ref{th:EULoverZ}, 
the assumption that $r$ and $s$ are positive cannot be dropped for the uniqueness. 
In fact, it is known that there exist two isomorphism classes of 
even unimodular lattices of signature $(16, 0)$, and 
$24$ isomorphism classes of even unimodular lattices of signature $(24, 0)$. 
For these facts, we refer to \cite[Chapter V, \S2.3]{Se73}.

\subsection{Symmetric polynomials}
Most of polynomials treated in this article are monic and have nonzero constant terms. 
We may sometimes assume that a factor of a monic polynomial is monic without mentioning. 
Let $K$ be a field. 

\begin{definition}
Let $F(X)\in K[X]$ be a monic polynomial with $F(0) \neq 0$. 
We define a monic polynomial $F^*(X)\in K[X]$ by 
\[ F^*(X) : = F(0)^{-1} X^{\deg F}  F(X^{-1}).  \]
If $F(X) = X^n + a_{n-1} X^{n-1} + \cdots + a_1 X + a_0$ ($a_0 \neq 0$) then
\[ F^*(X) = a_0^{-1} (a_0 X^n + a_1 X^{n-1} + \cdots + a_{n-1}X + a_n). \]
We say that $F$ is \textit{$*$-symmetric} if $F^* = F$. 
For a $*$-symmetric polynomial $F$, we have $F(0) = 1$ or $-1$. 
A $*$-symmetric polynomial $F$ is said to be \textit{$+1$-symmetric} 
(resp. \textit{$-1$-symmetric}) if $F(0) = 1$ (resp. $F(0) = -1$).
\end{definition}

Let $F\in K[X]$ be a monic polynomial. 
It can be checked that
if $F$ is $+1$-symmetric and has odd degree then $(X+1)\mid F$; and 
if $\ch(K)\neq 2$ and $F$ is $-1$-symmetric then $(X-1)\mid F$. 
These implies that every $*$-symmetric irreducible polynomial other than 
$X-1$ and $X+1$ is $+1$-symmetric and has even degree. 

\begin{definition}\label{def:type012}
We say that a $*$-symmetric polynomial $f\in K[X]$ is of
\begin{itemize}
\item \textit{type $0$} if $f$ is a product of powers of $(X-1)$ and of $(X+1)$;
\item \textit{type $1$} if $f$ is a product of $+1$-symmetric 
irreducible monic polynomials of even degrees;
\item \textit{type $2$} if $f$ is a product of polynomials of the form $gg^*$, where
$g$ is monic, irreducible and $g^* \neq g$.
\end{itemize}
Note that if $f$ is of type $2$ then $f$ is $+1$-symmetric and of even degree 
as well as the type $1$ case. 
For a monic polynomial $F\in K[X]$, we write $I_i(F;K)$ for the set of its irreducible 
factors of type $i$ over $K$ ($i = 0,1$), 
and define $I(F;K) := I_0(F;K) \cup I_1(F;K)$. 
The symbol $I_2(F;K)$ denotes the set of non-$*$-symmetric irreducible factors of $F$
in $K[X]$. 
\end{definition}

\begin{proposition}\label{prop:factorization_into_type012}
Let $F\in K[X]$ be a $*$-symmetric polynomial. 
For any irreducible monic polynomial $g\in K[X]$,  
the multiplicity of $g^*$ in $F$ is equal to that of $g$. 
As a result, $F$ can be expressed as 
\[ F 
= \prod_{f\in I_0(F;K)} f^{m_f} 
\times \prod_{f\in I_1(F;K)} f^{m_f} 
\times \prod_{\{g,g^*\}\subset I_2(F;K)} (gg^*)^{m_g}, 
\]
where $m_f$ is the multiplicity of $f\in K[X]$ in $F$.
\end{proposition}
\begin{proof}
Let $g$ be an irreducible monic polynomial. If $g = g^*$ then the assertion is clear. 
Suppose that $g \neq g^*$, and let $m_g, m_{g^*}$ be the multiplicities of 
$g,g^*$ respectively. We may assume that $m_{g^*} \geq m_g$ since $g^{**} = g$. 
Because $g$ and $g^*$ have no common factor, we can write 
\[ F 
= g^{m_g}(g^*)^{m_{g*}} H
= (gg^*)^{m_g} (g^*)^{m_{g*} - m_g} H
\]
for some monic polynomial $H\in K[X]$. 
Since $F$ and $(gg^*)^{m_g}$ are $*$-symmetric, it follows 
that $(g^*)^{m_{g*} - m_g} H$ is also $*$-symmetric, i.e., 
$g^{m_{g*} - m_g} H^* = (g^*)^{m_{g*} - m_g} H$. 
In particular $g^{m_{g*} - m_g}\mid (g^*)^{m_{g*} - m_g} H$. 
However, both $g^*$ and $H$ are coprime to $g$. Hence $m_{g*} - m_g = 0$. 
This completes the proof. 
\end{proof}

\begin{definition}\label{def:type_i_component}
Let $F\in K[X]$ be a $*$-symmetric polynomial. 
For each $i = 0,1,2$, the factor $F_i := \prod_{f\in I_i(F;K)}f^{m_f}$
of $F$ is referred to as the \textit{type $i$ component} of $F$ in $K[X]$. 
\end{definition}

Let $F\in K[X]$ be a $*$-symmetric polynomial. 
It is clear that $F_0$ and $F_1$ are $*$-symmetric polynomials of type $0$
and of type $1$ respectively. Furthermore, 
we have $F_2 = \prod_{\{g,g^*\}\subset I_2(F;K)} (gg^*)^{m_g}$
by Proposition \ref{prop:factorization_into_type012}, and it is 
actually a $*$-symmetric polynomial of type $2$. 
The factorization $F = F_0F_1F_2$ depends on the field being considered. 
For example, the polynomial $X^2 + 1$ is irreducible and of type $1$ in $\bQ[X]$, 
but it is factorized as $X^2 + 1 = (X - \sqrt{-1})(X + \sqrt{-1})$ and of type $2$ 
in $\bQ(\sqrt{-1})[X]$. On the other hand, $X-1$ and $X+1$ are of type $0$ over 
any field. It can be seen that 
if a $*$-symmetric polynomial $F\in K[X]$ is of type $2$ over $K$ then 
it is also type $2$ over any extension field $L$ of $K$. 

\begin{definition}\label{def:tracepl}
For a $+1$-symmetric polynomial $f\in K[X]$ of even degree $2n$, 
there exists a unique polynomial $h\in K[X]$ of degree $n$ such that 
$f(X) = X^n h(X + X^{-1})$. The polynomial $h$ is called the 
\textit{trace polynomial} of $f$. 
\end{definition}

\subsection{Isometries of inner product spaces}
Let $(V,b)$ be an inner product space over a field $K$. 
A $K$-linear transformation $t:V\to V$ is referred to as an \textit{isometry} of $(V,b)$ 
if $b(t(x), t(y)) = b(x, y)$ for any $x,y\in V$. 
Let $t$ be an isometry of $V$, and $F\in K[X]$ the characteristic polynomial of $t$. 
For a polynomial $f\in K[X]$, we define
\begin{equation*}
 V(f;t) := \{ v\in V \mid f(t)^N.v = 0 \quad\text{for some $N\in \bZ_{\geq 0}$} \}.   
\end{equation*}
Note that each $V(f;t)$ is $t$-stable, i.e., $t.V(f;t) \subset V(f;t)$. 
Moreover, we have $V(f;t)$ is orthogonal to $V(g;t)$ unless $g = f^*$, see 
\cite[Lemma 3.1]{Mi69}. It can be shown that $F$ is $*$-symmetric, and 
$V$ admits the following orthogonal direct sum decomposition:
\begin{equation}\label{eq:ODSD}
V = \bigoplus_{f\in I(F;K)} V(f;t)
\oplus \bigoplus_{\{g, g^*\}\subset I_2(F;K)} V(gg^*;t).  
\end{equation}
Since $V(gg^*;t) = V(g;t) + V(g^*;t)$ (direct sum) and $V(g;t) \perp V(g^*;t)$, 
the subspace $V(gg^*;t)$ is \textit{metabolic} i.e., it has a basis of which the 
Gram matrix is of the form 
\[ \begin{pmatrix}
O & \Id_n \\
\Id_n & O 
\end{pmatrix}, \]
where $\Id_n$ is the identify matrix.  

\medskip

\begin{definition}\label{def:associated_module}
Let $F\in K[X]$ be a $*$-symmetric polynomial. 
For a factor $f$ which is in $I(F;K)$ or of the form $gg^*$ for some $g\in I_2(F;K)$, 
we define $M^f := (K[X]/(f))^{m_f}$ where $m_f$ is the multiplicity of $f$ in $F$. 
Let us consider the $K$-algebra 
\begin{equation}\label{eq:associated_module_of_*-symm}
 M := \prod_{f\in I(F;K)} M^f 
\times \prod_{\{g,g^*\}\subset I_2(F;K)} M^{gg^*}. 
\end{equation}
Let $\alpha:M\to M$ be the linear transformation defined by 
the multiplication by $X$. Then, it is a semisimple transformation with 
characteristic polynomial $F$. 
When regarding $M$ as a $K[X]$-module by the action of $\alpha$, we refer to $M$ as 
the \textit{associated $K[X]$-module of $F$} with transformation $\alpha$. 
\end{definition}

\begin{remark}
We will consider inner products on $M$ making $\alpha$ an isometry. 
When discussing a general theory, it is convenient to regard $M$ as a 
$K[\Gamma]$-module rather than a $K[X]$-module, 
where $\Gamma$ is an infinite cyclic group. 
One reason is that $K[\Gamma]$ admits the involution determined by $g\mapsto g^{-1}$
for all $g\in \Gamma$. We refer to \cite[\S\S 3 and 4]{BT20}.  
\end{remark}

\subsection{Indices of isometries}
Here, we introduce an invariant of an isometry of an inner product space over $\bR$
called the index. For more detail, we refer to \cite[\S2.5]{Ta23} and references therein. 
Let $(V,b)$ be an inner product space over $\bR$. 
If $(r,s)$ is the signature of $(V,b)$ then the difference $r-s$ is called the 
\textit{index}. 
Let $t$ be an isometry of $(V,b)$ and $F\in \bR[X]$ its characteristic polynomial. 
We remark that the index of the subspace $V(gg^*;t)$ for $g\in I_2(F;\bR)$ is 
zero since $V(gg^*;t)$ is metabolic. 

\begin{definition}
The map $\idx_t : I(F;\bR) \to \bZ$ defined by 
\[ \idx_t(f) = \text{the index of $V(f;t)$} \quad (f\in I(F;\bR)) \]
is called the \textit{index of $t$} (with respect to $b$). 
It is sometimes written by $\idx_t^b$ if the inner product $b$ needs to be emphasized. 
\end{definition}
Let $f\in I(F, \bR)$, and let $m_f$ denote the multiplicity of $f$ in $F$. 
Note that the signature of $V(f;t)$ is given by 
\[ \left(
\frac{\deg(f^{m_f}) + \idx_t(f)}{2}, \frac{\deg(f^{m_f}) - \idx_t(f)}{2} \right). \]
We have
\begin{equation}\label{eq:idx1}
\text{
$-\deg(f^{m_f}) \leq \idx_t(f) \leq \deg(f^{m_f})$
and $\idx_t(f) \equiv \deg(f^{m_f}) \bmod 2$, }
\end{equation}
and 
\begin{equation}\label{eq:idx2}
\text{if $f \in I_1(F;\bR)$ then $(\deg(f^{m_f}) + \idx_t(f))/2 
\equiv (\deg(f^{m_f}) - \idx_t(f))/2 
\equiv 0 \bmod 2$} 
\end{equation}
(actually $\deg(f^{m_f}) = 2m_f$ because any $+1$-symmetric 
irreducible polynomial over $\bR$ other than $X+1$ is of the form 
$X^2 - (\delta + \delta^{-1})X + 1$ for some complex number $\delta \neq 1, -1$ 
with $|\delta| = 1$, and hence has degree $2$). 
Furthermore
\begin{equation}\label{eq:idx3}
\sum_{f\in I(F;\bR)} \fraci(f) = \text{the index of $(V,b)$} 
\end{equation}
by the orthogonal direct sum decomposition \eqref{eq:ODSD}. 

\begin{theorem}\label{th:condition_Sign}
Let $t$ be an isometry of an inner product space $(V,b)$ over 
$\bR$ of signature $(r,s)$. 
Let $F\in \bR[X]$ denote the characteristic polynomial of $t$, and 
$m(F)$ the number of roots of $F$ whose absolute values are greater than $1$
counted with multiplicity. Then 
\begin{equation}\label{eq:Sign}
\text{$r,s \geq m(F)$ and if $F(1)F(-1)\neq 0$ then $r\equiv s\equiv m(F) \bmod 2$}. 
\tag{Sign}  
\end{equation}
\end{theorem}
\begin{proof}
Let $(r_+, s_+)$ and $(r_-, s_-)$ denote the signatures of the subspaces 
$V(X-1;t)$ and $V(X+1;t)$ respectively, and  $(r_f, s_f)$ the signature of 
$V(f;t)$ for each $f\in I_1(F;\bR)$ or $f = gg^*$ where $g\in I_2(F;\bR)$. 
Then we have 
\[
r = r_+ + r_- + \sum_{f\in I_1(F;\bR)}r_f +  \sum_{\{g,g^*\}\subset I_2(F;\bR)}r_{gg^*}
\]
by the orthogonal direct sum decomposition \eqref{eq:ODSD}. 
Furthermore $r_{gg^*} = \deg((gg^*)^{m_g})/2$ since $V(gg^*;t)$ has index $0$. 
Thus 
\[ \sum_{\{g,g^*\}\subset I_2(F;\bR)}r_{gg^*}
= \sum_{\{g,g^*\}\subset I_2(F;\bR)} \deg((gg^*)^{m_g})/2 
= m(F),
\]
and we get 
\[
r = r_+ + r_- + \sum_{f\in I_1(F;\bR)}r_f +  m(F). 
\]
In particular $r \geq m(F)$. 
Moreover, since $r_f\equiv 0 \bmod 2$ by \eqref{eq:idx2}, 
if $F(1)F(-1) \neq 0$ then $r_+ = r_- = 0$, and $r \equiv m(F) \bmod 2$. 
It follows similarly that $s\geq m(F)$ and if $F(1)F(-1)\neq 0$ then 
$s\equiv m(F) \bmod 2$. 
\end{proof}

The condition \eqref{eq:Sign} will be denoted by \eqref{eq:Sign}$_{r,s}$ if necessary. 
Changing from the situation so far, 
suppose that a $*$-symmetric polynomial $F\in \bR[X]$ is given. 
For a given map $\fraci : I(F;\bR) \to \bZ$, we consider 
when there exists an inner product space over $\bR$ and its semisimple isometry $t$ 
such that $F$ is the characteristic polynomial of $t$ and $\idx_t = \fraci$. 

\begin{definition}\label{def:indexmap}
Let $r$ and $s$ be non-negative integers. 
The symbol $\Idx(r,s;F)$ denotes the set of all maps from $I(F;\bR)$ to $\bZ$
satisfying the following three conditions:
\begin{equation}\label{eq:indexdata1}
-\deg(f^{m_f}) \leq \fraci(f) \leq \deg(f^{m_f}) \text{ and } 
\fraci(f) \equiv \deg(f^{m_f}) \bmod 2 \text{ for each $f\in I(F;\bR)$}.
\end{equation}
\begin{equation}\label{eq:indexdata2}
\frac{\deg(f^{m_f}) + \fraci(f)}{2} 
\equiv \frac{\deg(f^{m_f}) - \fraci(f)}{2} 
\equiv 0 \bmod 2 \text{ for each $f\in I_1(F;\bR)$}. 
\end{equation}
\begin{equation}\label{eq:indexdata3}
\sum_{f\in I(F;\bR)} \fraci(f) = r -s.
\end{equation}
Each map in $\Idx(r,s;F)$ is referred to as an \textit{index map}. 
\end{definition}

For any isometry $t$ of an inner product space over $\bR$ of signature $(r,s)$, 
with characteristic polynomial $F$, its index $\idx_t$ belongs to $\Idx(r,s;F)$
by \eqref{eq:idx1}, \eqref{eq:idx2}, and \eqref{eq:idx3}. 
Conversely, any index map $\fraci\in \Idx(r,s;F)$ is realized as the index of 
some isometry $t$. More precisely, the following theorem holds.

\begin{theorem}\label{th:Sign_implies}
Let $r,s$ be non-negative integers, and let $F\in \bR[X]$ be a $*$-symmetric polynomial
of degree $r + s$. Suppose that the condition \eqref{eq:Sign} holds. 
Then the set $\Idx(r,s;F)$ is not empty. Moreover, for any $\fraci \in \Idx(r,s;F)$, 
there exists an inner product $b$ on the associated $\bR[X]$-module $M$ of $F$ 
with transformation $\alpha$ such that it makes $\alpha$ an isometry with index
$\fraci$. As a result, we can write
\[ \Idx(r,s;F) = \left\{ \idx_t \;\middle|\; 
\begin{tabular}{l}
\text{$t$ is a semisimple isometry of an inner product space } \\
\text{of signature $(r,s)$ with characteristic polynomial $F$}
\end{tabular}
\right\}.
\]
\end{theorem}
\begin{proof}
See the proof of \cite[Proposition 8.1 (b)]{Ba15}. 
\end{proof}

\begin{corollary}\label{cor:conditionSign}
Let $r,s$ be non-negative integers, and let $F\in \bR[X]$ be a $*$-symmetric polynomial
of degree $r + s$. The set $\Idx(r,s;F)$ is not empty if and only if 
the condition \eqref{eq:Sign} holds. 
\end{corollary}
\begin{proof}
The if part is included in Theorem \ref{th:Sign_implies}. 
Suppose that $\Idx(r,s;F)\neq \emptyset$. Then, Theorem \ref{th:Sign_implies} shows that 
there exists a semisimple isometry of an inner product space of signature $(r,s)$
with characteristic polynomial $F$. Hence, the condition \eqref{eq:Sign} holds
by Theorem \ref{th:condition_Sign}. 
\end{proof}

We will need the following lemma. 

\begin{lemma}\label{lem:signature_of_F12(1)F12(-1)}
Let $r,s$ be non-negative integers with $r\equiv s \bmod 8$, and 
$F\in \bR[X]$ a $*$-symmetric polynomial of degree $r + s$ with 
the condition \eqref{eq:Sign}. We write $F(X) = (X-1)^{m_+}(X+1)^{m_-}F_{12}(X)$, 
where $m_\pm$ is the multiplicity of $X\mp1$ in $F$ and $F_{12}$ is the remaining factor. 
Let $\fraci \in \Idx(r,s;F)$ be an index map, and  
put $n = \deg(F)/2$, $s_\pm = (m_\pm - \fraci(X\mp1))/2$. 
Then the signature of $F_{12}(1)F_{12}(-1)$ is equal to $(-1)^{n + s_+ + s_-}$, 
or equivalently, we have  
\[ (-1)^n F_{12}(1)F_{12}(-1) 
= (-1)^{s_+}\, |F_{12}(1)|\cdot
(-1)^{s_-}\, |F_{12}(-1)|. 
\]
\end{lemma}
\begin{proof}
Let $M$ be the associated $\bR[X]$-module of $F$ with transformation $\alpha$. 
Then, there exists an inner product $b$ on $M$ which makes $\alpha$ 
an isometry with index $\fraci$ by Theorem \ref{th:Sign_implies}. 
Let $s_{12}$ denote the signature of the subspace 
$M^{12} := \prod_{f\in I_1(F;\bR)} M^f \oplus \prod_{\{g,g^*\}\subset I_2(F;\bR)} M^{gg^*}$ 
of $M = (M, b)$. 
Then $s_+ + s_- + s_{12} = s$ since $(M_\infty, b_\infty)$ has signature $(r,s)$. 
Furthermore, since $n\equiv (r+s)/2 \equiv s \bmod 2$ by $r\equiv s \bmod 8$, we get  
$(-1)^{n + s_+ + s_-} = (-1)^{s + s_+ + s_-} = (-1)^{s_{12}}$. 
On the other hand, the signature of $F_{12}(1)F_{12}(-1)$ is equal to
$\det(b|_{M^{12}})$ (see e.g. \cite[Corollary 5.2]{Ba15}), which is $(-1)^{s_{12}}$. 
Hence, the signature of $F_{12}(1)F_{12}(-1)$ is equal to $(-1)^{n + s_+ + s_-}$. 
\end{proof}

\section{Local-global principle and obstruction}\label{sec:LGPO}
In this section, we fix non-negative integers $r,s\in \bZ_{\geq 0}$
with $r\equiv s \bmod 8$ and a $*$-symmetric polynomial $F\in \bZ[X]$ of degree $r+s$. 
Note that the congruence $r\equiv s \bmod 8$ is a necessary and sufficient condition 
for the existence of an even unimodular lattice of signature $(r,s)$
(Theorem \ref{th:EULoverZ}). 
We refer to an isometry with characteristic polynomial $F$ and index $\fraci$
as an \textit{$(F,\fraci)$-isometry} for short. 
Bayer-Fluckiger \cite{Ba21, Ba22} gave a necessary and sufficient condition for 
the existence of an even unimodular lattice having a semisimple $(F, \fraci)$-isometry. 
It was reformulated and extended to the case where $F$ is $*$-symmetric, 
which covers the $-1$-symmetric case, in \cite{Ta23}. 
We explain the criterion here.

\subsection{Local conditions}\label{LocalConditions}
If $F$ is the characteristic polynomial of a semisimple isometry $t$ of 
an even unimodular lattice $(\Lambda, b)$ of signature $(r, s)$, 
then $F$ is also the characteristic polynomial of a semisimple isometry $t$ of
the inner product space $(\Lambda\otimes \bR, b \otimes \bR)$ over $\bR$ of 
signature $(r, s)$. Thus, Theorem \ref{th:condition_Sign} shows that
$F$ must satisfy the condition \eqref{eq:Sign}$_{r,s}$. 
Namely, the localization at the infinite place yields the condition 
\eqref{eq:Sign}$_{r,s}$. 
So, what condition is produced by the localizations at finite places?
The answer is as follows. 

\begin{proposition}\label{prop:conditionSquare}
Suppose that $F$ is the characteristic polynomial of a semisimple isometry $t$ of 
an even unimodular lattice $(\Lambda, b)$ over $\bZ$. Then 
$\deg(F)$ is even and 
\begin{equation}\label{eq:Square}
\text{$|F(1)|, |F(-1)|$ and $(-1)^{(\deg F)/2} F(1) F(-1)$ are all squares.} \tag{Square}
\end{equation}
\end{proposition}
\begin{proof}
See \cite[Corollary 4.3]{Ta23}. 
\end{proof}

For the condition \eqref{eq:Square}, we have the following lemma. 

\begin{lemma}\label{lem:properties_for_Square}
The following assertions hold.  
\begin{enumerate}
\item Let $f$ and $g\in \bZ[X]$ be $*$-symmetric polynomials of even degrees. 
If $f$ and $g$ satisfy \eqref{eq:Square} then so does $fg$. 
If $g$ and $fg$ satisfy \eqref{eq:Square}, and $g(1)g(-1)\neq 0$, 
then $f$ satisfies \eqref{eq:Square}. 
\item Any $*$-symmetric polynomial in $\bZ[X]$ of type $2$ over $\bQ$ satisfies 
\eqref{eq:Square}. 
\end{enumerate}
\end{lemma}
\begin{proof}
(i). Straightforward. 

(ii). Let $f\in \bZ[X]$ be a $*$-symmetric polynomial of type $2$. 
Then, there exists a monic polynomial $g\in \bZ[X]$ such that $f = gg^*$. 
Note that $g(0)$ is a unit, i.e., $g(0) = 1$ or $-1$, because $g(0)$ and 
$g(0)^{-1}$ are the constant terms of $g$ and $g^* \in \bZ[X]$ respectively. 
Then 
\[ |f(\pm1)| 
= |g(\pm1)g^*(\pm1)|
= |g(\pm1) g(0)^{-1} (\pm1)^{\deg g} g(\pm1)|
= |g(\pm1)|^2.  
\]
Moreover, we have 
\[ \begin{split}
(-1)^{\deg(f)/2}f(1)f(-1)
&= (-1)^{\deg g} g(1)g^*(1) \cdot g(-1)g^*(-1) \\
&= (-1)^{\deg g} g(1)g(0)^{-1}g(1)\cdot g(-1)g(0)^{-1}(-1)^{\deg g}g(-1) \\
&= (g(1)g(-1))^2.
\end{split}
\]
This completes the proof. 
\end{proof}

\subsection{Local-global principle}
\begin{notation}\label{nt:LGPO}
We use the following notation in the sequel. 
\begin{enumerate}
\item The set of all places of $\bQ$ is denoted by $\cV$. 
%
\item For a monic polynomial $f$, we write $m_f$ for the multiplicity of 
$f$ in $F$. We often write $m_\pm$ for $m_{X\mp1}$ briefly. 
Furthermore, the sets $I(F;\bQ)$, $I_1(F;\bQ)$, and $I_2(F;\bQ)$ 
(see Definition \ref{def:type012}) are abbreviated to $I$, $I_1$, and $I_2$ respectively. 
The product of the type $1$ and $2$ components of $F$ is denoted by $F_{12}$. 
Under this notation, we can write 
\[ \begin{split}
F(X) &= (X-1)^{m_+} (X+1)^{m_-} F_{12}(X), \\
F_{12}(X) &= \prod_{f\in I_1}f(X)^{m_f}
\times \prod_{\{g,g^*\}\subset I_2}(g(X)g^*(X))^{m_g}. 
\end{split} 
\]
\item The symbol $M$ denotes the associated $\bQ[X]$-module of $F$ with 
transformation $\alpha$, and $M^f$ 
($f$ is a factor in $I$ or of the form $gg^*$ for some $g\in I_2$) is as in 
Definition \ref{def:associated_module}. We write $M^\pm = M^{X\mp1}$. 
\item For a place $v\in \cV$, we define $M_v := M\otimes\bQ_v$. 
Similarly $M^f_v := M^f\otimes\bQ_v$ for $f$
which is in $I$ or of the form $gg^*$ for some $g\in I_2$. Then 
\[ M_v = M^+_v \oplus M^-_v 
\oplus \bigoplus_{f\in I_1} M^f_v 
\oplus \bigoplus_{\{g,g^*\} \subset I_2} M^{gg^*}_v.  
\]
Note that $\alpha$ is extended to a $\bQ_v$-linear transformation on $M_v$ in a unique 
way, and $M_v$ is (isomorphic to) the associated $\bQ_v[X]$-module of $F$.  
\end{enumerate}
\end{notation}

We consider when there exists an inner product $b$ on $M$ 
such that $\alpha$ becomes an isometry having a given index $\fraci$ 
and $(M,b)$ contains an $\alpha$-stable even unimodular lattice of signature $(r,s)$. 
The following notation concerns localizations of this question. 

\begin{notation}\label{nt:P1P2P3}
Let $\fraci\in \Idx(r,s;F)$ be an index map. 
We consider the following three properties \eqref{eq:P1}--\eqref{eq:P3} of an inner 
product $b_v$ on $M_v$ for each $v\in \cV$.
The first property is that
\begin{equation}\label{eq:P1}
\text{$\alpha:M_v \to M_v$ is an isometry with respect to $b_v$.} \tag{P1}
\end{equation}
Assume that $b_v$ has the property \eqref{eq:P1}. The second property is that
\begin{equation}\label{eq:P2}
\begin{split}
&\text{if $v\neq \infty$ then there exists an $\alpha$-stable
even unimodular lattice over $\bZ_v$ on $(M_v, b_v)$, and}\\
&\text{if $v = \infty$ then the isometry $\alpha$ of $(M_\infty, b_\infty)$ 
has index $\fraci$.}
\end{split}\tag{P2}
\end{equation}
The last property is that
\begin{equation}\label{eq:P3}
\text{$\det(b_v|_{M_v^\pm}) = \delta_\pm$ in $\bQ_v^\times/\bQ_v^{\times 2}$,} \tag{P3}
\end{equation}
where $\delta_+$ and $\delta_- \in \bQ^\times$ are nonzero rational numbers defined by 
\[ \delta_\pm := \begin{cases}
(-1)^{(m_\pm - \fraci(X\mp1))/2}\, |F_{12}(\pm1)| &\text{if $m_+$ is even} \\
(-1)^{(m_\pm - \fraci(X\mp1))/2}\, 2|F_{12}(\pm1)| &\text{if $m_+$ is odd.}
\end{cases}\]
Moreover, we write $\cB_\fraci$ for the set of families $\{b_v\}_{v\in \cV}$ of inner products 
on $M_v$ such that each $b_v$ has the properties \eqref{eq:P1}--\eqref{eq:P3} and 
$\#\{ v \in \cV \mid \hw_v(b_v|_{M^f_v}) \neq 0 \}$ is finite for all $f\in I$. 
\end{notation}


\begin{proposition}
If $F$ satisfies the conditions \eqref{eq:Sign}$_{r,s}$ and \eqref{eq:Square} 
then $\cB_\fraci$ is not empty for any $\fraci\in \Idx(r,s;F)$. 
\end{proposition}
\begin{proof}
See \cite[Proposition 6.1]{Ta23}. 
\end{proof}

If $b$ is an inner product on $M$ such that $\alpha:M\to M$ becomes an isometry 
having index $\fraci$ and $(M,b)$ contains an $\alpha$-stable even unimodular lattice 
over $\bZ$, then the family $\{b\otimes \bQ_v\}_{v\in \cV}$ of inner products 
$b\otimes \bQ_v :M_v \times M_v \to \bQ_v$ obtained by localizations 
belongs to $\cB_\fraci$. 
The local-global principle for the existence of such an inner product on $M$ 
is described as follows. 

\begin{theorem}\label{th:LGP}
Let $r,s\in \bZ_{\geq 0}$ be non-negative integers with $r\equiv s \bmod 8$, 
$F\in \bZ[X]$ a $*$-symmetric polynomial of degree $r + s$ with the conditions 
\eqref{eq:Sign}$_{r,s}$ and \eqref{eq:Square}, and 
$\fraci\in\Idx(r,s;F)$ an index map. The following conditions are equivalent:
\begin{enumerate}
\item There exists an inner product $b$ on $M$ such that 
$\alpha:M\to M$ becomes an isometry having index $\fraci$ and  
$(M,b)$ contains an $\alpha$-stable even unimodular lattice over $\bZ$. 
\item There exists a family $\{b_v\}_{v\in \cV}\in \cB_\fraci$ such that 
$\sum_{v\in \cV} \hw_v(b_v|_{M^f_v}) = 0$ for any $f\in I$. 
\end{enumerate}
\end{theorem}
\begin{proof}
See \cite[Proposition 5.2]{Ta23}. 
\end{proof}

\subsection{Local-global obstruction}
In the rest of this section, we assume that $F$ satisfies the conditions 
\eqref{eq:Sign}${}_{r,s}$ and \eqref{eq:Square}, and fix an index map 
$\fraci \in \Idx(r,s;F)$.  
We rephrase the condition (ii) of Theorem \ref{th:LGP} further. 

\begin{notation}
Let $C(I)$ denote the $\bZ/2\bZ$-module consisting of all maps from 
$I = I(F;\bQ)$ to $\bZ/2\bZ$, that is,  
$C(I) := \{ \gamma: I\to \bZ/2\bZ \} = (\bZ/2\bZ)^{\oplus I}$.  
Moreover, we define a map $\eta : \cB_\fraci \to C(I)$ by
\[ \eta(\beta)(f) = \sum_{v\in \cV}\hw_v(b_v|_{M^f_v}) \in \bZ/2\bZ
\quad\textstyle (\beta = \{b_v\}_{v\in \cV}\in \cB_\fraci, \,f\in I ).
\]
\end{notation}

Under this notation, the condition (ii) in Theorem \ref{th:LGP} is equivalent to 
$\bm{0} \in \eta(\cB_\fraci)$, where $\bm{0}:I \to \bZ/2\bZ$ is the zero map. 
We explain that the image $\eta(\cB_\fraci)$  of $\eta$ coincides with a coset of 
some submodule in $C(I)$.

\begin{notation}\label{nt:Pi_fg}
For a monic polynomial $f\in \bZ[X]$, the symbol $\overline{I(f;\bQ_p)}$ denotes 
the set of irreducible factors of reductions modulo $p$ of polynomials in $I(f;\bQ_p)$: 
\[ \overline{I(f;\bQ_p)} := 
\left\{ \bar{h} \in \bF_p[X] \;\middle|\; 
\begin{tabular}{l}
\text{$\bar{h}$ is irreducible, and there exists a $*$-symmetric } \\
\text{irreducible factor of $f$ in $\bZ_p[X]$ whose reduction } \\
\text{modulo $p$ is divisible by $\bar{h}$ in $\bF_p[X]$}
\end{tabular}
\right\}.
\]
Moreover, we define
\[ \overline{I(X\mp 1;\bQ_p)}' := \begin{cases}
\overline{I(X\mp 1;\bQ_p)} = \{ X\mp 1 \} 
& \text{if $m_\pm \geq 3$; or $m_\pm = 2$ and 
$\delta_\pm \neq -1 \in \bQ_p^\times/\bQ_p^{\times 2}$} \\
\emptyset & \text{otherwise}
\end{cases}
\]
($m_\pm$ and $\delta_\pm$ are defined in Notations \ref{nt:LGPO} and \ref{nt:P1P2P3})
and $\overline{I(f;\bQ_p)}' := \overline{I(f;\bQ_p)}$ for a monic 
polynomial with $f(1)f(-1)\neq 0$. 
Note that the set $\overline{I(X\mp 1;\bQ_p)}'$ depends on the polynomial 
$F$ and the value $\fraci(X\mp1)$ since so do $m_\pm$ and $\delta_\pm$. 
We define a set $\Pi_\fraci^F(f,g)$ of primes for monic polynomials
$f,g\in \bZ[X]$ by 
\[ \Pi_\fraci^F(f,g) := \{ p:\text{prime} 
\mid \overline{I(f;\bQ_p)}'\cap \overline{I(g;\bQ_p)}' \neq \emptyset \}. \]
\end{notation}

\begin{remark}\label{rem:for_Pi_fg}
We give some remarks on Notation \ref{nt:Pi_fg}. Let $p$ be a prime. 
\begin{enumerate}
\item The set $\overline{I(f;\bQ_p)}$ is a subset of $I(f \bmod p; \bF_p)$, but 
they do not necessarily coincide. 
For example, let us consider the case $f(X) = X^2 - 11X + 1$ 
and $p = 3$. Then $f$ decomposes as 
\[ f(X) = (X - (11+3\sqrt{13})/2)(X - (11-3\sqrt{13})/2) \quad\text{in $\bQ_3[X]$} \]
because $13$ is a square in $\bZ_3$. Thus $f$ is of type $2$ in $\bQ_3[X]$ and 
$\overline{I(f;\bQ_3)}= \emptyset$. On the other hand, we have 
$I(f \bmod 3; \bF_3) = \{X-1\}$ since $f(X) \bmod 3 = X^2 -2X + 1 = (X-1)^2$.  
\item The condition for $\overline{I(X\mp 1;\bQ_p)}'$ to be not empty, 
$m_\pm \geq 3$ or $m_\pm = 2$ and 
$\delta_\pm \neq -1 \in \bQ_p^\times/\bQ_p^{\times 2}$, 
guarantees that there exist inner products $b^\pm_p$ and $\widehat{b}^\pm_p$
on $M^\pm_p$ such that $\det(b^\pm_p) = \det(\widehat{b}^\pm_p) = \delta_\pm$
and $\hw_p(b^\pm_p)\neq \hw_p(\widehat{b}^\pm_p)$, 
see \cite[Chapter IV, Proposition 6]{Se73}.
\end{enumerate}
\end{remark}

For a subset $J$ of $I$, we write $\bm1_J \in C(I)$ for the characteristic function: 
\[ \bm1_J(f) = \begin{cases}
  1 & \text{if $f\in J$} \\
  0 & \text{if $f\notin J$}
\end{cases}
\quad (f\in I).
\]
Let $C_0$ denote the submodule of $C(I)$ generated by the subset 
\[
\{ \bm1_{\{f, g\}} \mid \text{$f,g \in I$ are distinct factors with 
$\Pi_\fraci^F(f,g) \neq \emptyset$} \}. 
\]

\begin{theorem}\label{th:imageofeta}
The image $\eta(\cB_\fraci)$ coincides with a coset of $C_0$. 
\end{theorem}
\begin{proof}
See \cite[Theorem 6.3]{Ta23}. 
\end{proof}

\begin{definition}\label{def:ERdefinedby}
The \textit{equivalence relation defined by $(F, \fraci)$}, written $\sim$, 
is the one on $I$ generated by the binary relation
$\{ (f,g) \in I\times I \mid \Pi_\fraci^F(f,g)\neq \emptyset \}$. 
We write $\Omega_\fraci$ for the submodule 
$\{c \in C(I) \mid \text{ $c(f) = c(g)$ if $f \sim g$} \}$ 
of $C(I)$. This is the submodule consisting of maps which are 
constant on each equivalence class with respect to $\sim$.  
(We remark that the symbol $\sim$ in \cite{Ta23} is different from the one here.)
\end{definition}

We now introduce an inner product $C(I)\times C(I) \to \bZ/2\bZ$ by 
\[ 
\gamma\cdot c = \sum_{f\in I}\gamma(f)c(f) \quad(\gamma, c \in C(I)). 
\]
Then the two submodules $C_0$ and $\Omega_\fraci$ of $C(I)$ are related as follows. 

\begin{lemma}\label{lem:OG=C0perp}
We have $\Omega_\fraci = C_0^\perp$, or equivalently  $\Omega_\fraci^\perp = C_0$.
\end{lemma}
\begin{proof}
See \cite[Lemma 6.10]{Ta23}. 
\end{proof}

\begin{definition}
The homomorphism 
\[ \Omega_\fraci \to \bZ/2\bZ, \, c \mapsto \eta(\beta)\cdot c  \]
is defined independently of the choice of $\beta\in \cB_\fraci$ since 
$\eta(\cB_\fraci)$ is a coset of $C_0$ by Theorem \ref{th:imageofeta} 
and $\Omega_\fraci \subset C_0^\perp$ by Lemma \ref{lem:OG=C0perp}. This homomorphism 
is called the \textit{obstruction map} for $(F, \fraci)$ and denoted by 
$\ob_\fraci : \Omega_\fraci \to \bZ/2\bZ$. 
The submodule $\Omega_\fraci \subset C(I)$ is called the \textit{obstruction group} 
for $(F, \fraci)$. Note that the obstruction group $\Omega_\fraci$ does 
not depend on $\fraci$ unless $m_+ = 2$ or $m_- = 2$ since so does $\sim$. 
\end{definition}

It is obvious that the obstruction map $\ob_\fraci : \Omega_\fraci \to \bZ/2\bZ$ 
is zero if the image $\eta(\cB_\fraci)$ of $\eta$ contains $\bm{0}$. 
We show that the converse is also true. 
Before that, let us see that the obstruction map factors through the 
quotient module $\Omega_\fraci/\{ \text{constant maps}\}$.  

\begin{proposition}\label{prop:even1s}
We have $\eta(\beta)\cdot \bm1_I = 0$ for any $\beta\in \cB_\fraci$. 
\end{proposition}
\begin{proof}
See \cite[Proposition 6.7]{Ta23}. 
\end{proof}

\begin{definition}
We refer to the quotient module 
$\widetilde\Omega_\fraci := \Omega_\fraci/\{ \bm0, \bm1_I \}$
as the \textit{reduced obstruction group} for $(F, \fraci)$. 
The obstruction map factors through $\widetilde\Omega_\fraci$ by 
Proposition \ref{prop:even1s}. The induces homomorphism 
$\widetilde\Omega_\fraci \to \bZ/2\bZ$ is referred to as the 
\textit{reduced obstruction map} for $(F, \fraci)$ and denoted 
$\widetilde\ob_\fraci : \Omega_\fraci \to \bZ/2\bZ$. 
\end{definition}

\begin{remark}
In \cite{Ta23}, we refer to the reduced obstruction group (resp. map) as 
the obstruction group (resp. map). 
\end{remark}

The following theorem is the full version of Theorem \ref{th:A}. 

\begin{theorem}\label{th:obstruction}
Let $r,s\in \bZ_{\geq 0}$ be non-negative integers with $r\equiv s \bmod 8$, 
$F\in \bZ[X]$ a $*$-symmetric polynomial of degree $r+s$ with the conditions 
\eqref{eq:Sign}$_{r,s}$ and \eqref{eq:Square}, 
and $\fraci\in \Idx(r,s;F)$ an index map. 
The following conditions are equivalent:
\begin{enumerate}
\item There exists an even unimodular lattice over $\bZ$ of signature $(r,s)$ 
having a semisimple $(F, \fraci)$-isometry. 
\item The associated $\bQ[X]$-module $M$ of $F$ with transformation $\alpha$ 
admits an inner product $b$ such that $\alpha$ becomes an isometry having index 
$\fraci$ and $(M,b)$ contains an $\alpha$-stable even unimodular lattice over $\bZ$. 
\item There exists a family $\beta\in \cB_\fraci$ such that $\eta(\beta) = \bm0$. 
\item The obstruction map $\ob_\fraci:\Omega_\fraci \to \bZ/2\bZ$ is the zero map. 
\item The reduced obstruction map $\widetilde\ob_\fraci:\widetilde\Omega_\fraci \to \bZ/2\bZ$ 
is the zero map.
\end{enumerate}
Moreover, if both $r$ and $s$ are positive then the 
following condition is also equivalent: 
\begin{enumerate}\setcounter{enumi}{5}
\item Any even unimodular lattice of signature $(r,s)$ admits a 
semisimple $(F,\fraci)$-isometry.  
\end{enumerate}
\end{theorem}
\begin{proof}
Suppose that there exists an even unimodular lattice $(\Lambda,b)$ over $\bZ$ of 
signature $(r,s)$ having a semisimple $(F, \fraci)$-isometry $t$. 
By identifying $\Lambda\otimes \bQ$ with $M$ and 
$t$ with $\alpha$, we obtain (i) $\Rightarrow$ (ii). 
The reverse implication (ii) $\Rightarrow$ (i) is obvious. 
The equivalence (ii) $\Leftrightarrow$ (iii) is nothing but Theorem \ref{th:LGP}. 
(iii) $\Rightarrow$ (iv) is clear. We prove (iv) $\Rightarrow$ (iii). 
Suppose that the obstruction map $\ob_\fraci:\Omega_\fraci \to \bZ/2\bZ$ vanishes, 
and take $\widetilde\beta\in \cB_\fraci$ arbitrarily. Then 
$\eta(\widetilde\beta)\in \Omega_\fraci^\perp = C_0$ by Lemma \ref{lem:OG=C0perp}. 
Thus $\eta(\cB_\fraci) = C_0$ by Theorem \ref{th:imageofeta}, and 
there exists $\beta\in \cB_\fraci$ such that $\eta(\beta) = \bm0$. 
This shows (iv) $\Rightarrow$ (iii). 
The equivalence (iv) $\Leftrightarrow$ (v) is obvious. 
Moreover, if both $r$ and $s$ are positive then the equivalence 
(i) $\Leftrightarrow$ (vi) follows from the uniqueness of 
an even unimodular lattice of signature $(r,s)$ 
(Theorem \ref{th:EULoverZ}). The proof is complete. 
\end{proof}

\begin{remark}\label{rem:history_for_LGO}
We give some historical remarks on Theorem \ref{th:obstruction}. 
Let $F\in \bZ[X]$ be a $*$-symmetric polynomial of even degree. 
As mentioned in Introduction, Bayer-Fluckiger and Taelman \cite{BT20} proved that 
when $F$ is irreducible (or a power of a $*$-symmetric irreducible polynomial), 
the conditions \eqref{eq:Sign} and \eqref{eq:Square} are necessary and sufficient 
for the existence of an even unimodular lattice with a prescribed signature, 
having a semisimple isometry of characteristic polynomial $F$. 
For its proof, the local-global idea was introduced as well as the argument for 
the local existence of a unimodular lattice in terms of equivariant Witt groups. 
Afterwards, Bayer-Fluckiger \cite{Ba20,Ba21,Ba22} proceeded to the case where 
the polynomial $F$ is reducible and $+1$-symmetric, and proved that 
the equivalence (i) $\Leftrightarrow$ (v) of Theorem \ref{th:obstruction} when 
$F$ is $+1$-symmetric. However, the first proof given in her preprint \cite{Ba21} 
was difficult to follow, 
and did not take account of subtle conditions for obstruction that arise 
from the difference among the sets $\overline{I(f;\bQ_p)}$, $\overline{I(f;\bQ_p)}'$, 
and $I(f \bmod p;\bF_p)$ for a monic polynomial of $f\in \bZ[X]$, see 
Notation \ref{nt:Pi_fg} and Remark \ref{rem:for_Pi_fg}. 
On the framework established by her, the author \cite{Ta23} improved the outlook 
of the proof, for example, by giving the intermediate result, Theorem \ref{th:LGP}, 
between (i) and (v). He also modified the definition of obstruction.  
Moreover, in \cite{Ta23}, the result extended to the case where $F$ is $*$-symmetric, 
which covers the $-1$-symmetric case, mainly by careful analysis at the prime $2$. 
\end{remark}

The following is a typical case where the obstruction vanishes. 
We say that an equivalence relation on $I$ is \textit{weakest} if  all elements of $I$ 
are equivalent one another. 

\begin{theorem}\label{th:weakestER}
There exists an even unimodular lattice over $\bZ$ of signature $(r,s)$ 
having a semisimple $(F, \fraci)$-isometry if 
the equivalence relation on $I$ defined by $(F, \fraci)$ is weakest. 
\end{theorem}
\begin{proof}
Suppose that equivalence relation on $I$ defined by $(F, \fraci)$ is weakest. 
This means that $\Omega_\fraci$ consists of the constant maps, and 
$\widetilde{\Omega}_\fraci$ is a trivial group. 
Hence, the reduced obstruction map 
$\widetilde\ob_\fraci:\widetilde\Omega_\fraci \to \bZ/2\bZ$ is zero, which 
implies that there exists an even unimodular lattice over $\bZ$ of signature $(r,s)$ 
having a semisimple $(F, \fraci)$-isometry by Theorem \ref{th:obstruction}. 
\end{proof}

For example, if $F$ is a power of a $*$-symmetric irreducible polynomial $f\in \bZ[X]$
then $I = \{f\}$ and the equivalence relation defined by $(F, \fraci)$ is 
clearly weakest. In this case, the obstruction vanishes independently of the index 
map $\fraci$.

\section{Computation of obstruction}\label{sec:Comp_of_Ob}
This section gives some results for computing obstruction. 
In particular, Theorem \ref{th:comparison}, which compares two obstruction maps, 
together with results in \S \ref{sec:Iso_on_EULofIndex0}, provides a 
systematic way to compute an obstruction map.

\subsection{Computation of $\Pi(f,g)$}
Let $K$ be a field, and let  
$f(X) = \sum_{i=0}^m a_i X^i$ and $g(X) = \sum_{j=0}^n b_j X^j \in K[X]$ 
be polynomials of degrees $m$ and $n$ respectively. The \textit{resultant} of 
$f$ and $g$ is the determinant of the $(m+n)\times(m+n)$ matrix
\[ \begin{pmatrix}
a_m & a_{m-1} & \cdots & a_0     &        &      \\
    & \ddots  & \ddots &         & \ddots &      \\
    &         & a_m    & a_{m-1} & \cdots & a_0  \\
b_n & b_{n-1} & \cdots & b_0     &        &      \\
    & \ddots  & \ddots &         & \ddots &      \\
    &         & b_n    & b_{n-1} & \cdots & b_0  \\
\end{pmatrix}, 
\]
and denoted by $\Res(f,g)$. 
If $\alpha_1, \ldots, \alpha_m$ and $\beta_1, \ldots, \beta_n\in \overline{K}$ 
are the roots of $f$ and $g$ respectively, then it is known that 
\begin{equation}\label{eq:resultant}
\Res(f,g) = a_m^n b_n^m \prod_{i=1}^m\prod_{j=1}^n(\alpha_i - \beta_j).   
\end{equation}
As a result, the resultant $\Res(f,g)$ vanishes if and only if $f$ and $g$ have 
a common root. If $f$ and $g$ are monic then Equation \eqref{eq:resultant} can also be  
written as 
\[ \Res(f,g) = \prod_{i = 1}^m g(\alpha_i) = (-1)^m \prod_{j=1}^n f(\beta_j).  \]

Let $f$ and $g\in \bZ[X]$ be monic polynomials with coefficients in $\bZ$. 
Note that $\Res(f,g)$ (considered as $f,g \in \bQ[X]$) is an integer by definition. 
Moreover, for any prime $p$ we have 
\[ \Res(f \bmod p, g \bmod p) \equiv \Res(f,g) \mod p.  \]
Thus $(f \bmod p)$ and $(g \bmod p)\in \bF_p[X]$ have a common factor if and only if 
$\Res(f,g) \equiv 0 \bmod p$. 

\begin{notation}\label{nt:Pi_fg_pure}
For two monic polynomials $f$ and $g\in \bZ[X]$, we define 
\[ \Pi(f,g) := \{ p:\text{prime} 
\mid \overline{I(f;\bQ_p)}\cap \overline{I(g;\bQ_p)} \neq \emptyset \}, \]
where the set $\overline{I(f;\bQ_p)}$ is defined in Notation \ref{nt:Pi_fg}. 
In the situation of Notation \ref{nt:Pi_fg}, 
we have $\Pi_\fraci^F(f,g) \subset \Pi(f,g)$ and they coincide 
if $f(1)f(-1)g(1)g(-1) \neq 0$. 
\end{notation}

The set $\Pi(f,g)$ and the resultant $\Res(f,g)$ are related as follows. 

\begin{proposition}\label{prop:Pi<Res}
Let $f$ and $g\in \bZ[X]$ be monic polynomials. Then 
\[ \Pi(f,g) 
\subset \{ p:\textup{prime} \mid \textup{$p$ is a factor of $\Res(f,g)$} \}. \]
\end{proposition}
\begin{proof}
Suppose that a prime $p$ belongs to $\Pi(f,g)$. 
Then $(f \bmod p)$ and $(g \bmod p)\in \bF_p[X]$ have a common factor, and thus 
$p\mid \Res(f,g)$. This shows the assertion. 
\end{proof}

\begin{example}\label{ex:Pi_cyclo12_X-1}
Let $f(X) = X^4 - X^2 + 1$; this is the $12$-th cyclotomic polynomial 
and irreducible (see \S\ref{ss:cyclotomic_modulo_p}). Then, 
the resultant $\Res(f, X-1) = -f(1) = -1$ has no prime factor. 
This means that $\Pi(f, X-1) = \emptyset$ by Proposition \ref{prop:Pi<Res}.  
\end{example}

We will give an explicit description of the set $\Pi(f,g)$ when $f$ and $g$ are 
cyclotomic polynomials in \S \ref{ss:Comp_Pi_cyclos}. 
The following proposition will be useful.  

\begin{proposition}\label{prop:nonSquare}
Let $f\in \bZ[X]$ be a $*$-symmetric polynomial with $f(1)f(-1)\neq 0$, and 
let $p$ be a prime.  
\begin{enumerate}
\item If $v_p(f(1))$ is odd then $X-1 \in \overline{I(f;\bQ_p)}$. 
\item If $v_p(f(-1))$ is odd then $X+1 \in \overline{I(f;\bQ_p)}$. 
\item If $(-1)^{\deg(f)/2}f(1)f(-1) \neq 1$ nor $-3$ in $\bQ_2^\times/\bQ_2^{\times 2}$ 
then $X-1 \in \overline{I(f;\bQ_2)}$. 
\end{enumerate} 
\end{proposition}
\begin{proof}
See \cite[Lemma 6.12]{Ta23}. 
\end{proof}

\subsection{Cyclotomic polynomials modulo $p$}\label{ss:cyclotomic_modulo_p}
In this subsection, we study factorizations of cyclotomic polynomials modulo a prime 
number. For a positive integer $n\in \bZ_{>0}$, we define
$R_n := \{ j\in \bZ \mid 1\leq j \leq n \text{ and } \gcd(j,n) = 1 \}$. 
The function $\varphi:\bZ_{>0} \to \bZ$ defined by 
$\varphi(n) = \# R_n$ is called \textit{Euler's totient function}. 
Let $n$ be a positive integer. 
The complex number $\zeta_n := \exp(2\pi\sqrt{-1}/n)$ is a 
primitive $n$-th root of unity. 
There are exactly $\varphi(n)$ primitive $n$-th roots of unity, 
and they are given by $\zeta_n^j$ for $j\in R_n$.
The polynomial $\prod_{j\in R_n}(X-\zeta_n^j)\in \bC[X]$ is called the 
\textit{$n$-th cyclotomic polynomial} and denoted by $\Phi_n(X)$. 
It is known that $\Phi_n(X)$ is a monic polynomial with coefficients in $\bZ$.  
Moreover, it is the minimal polynomial of $\zeta_n$ over $\bQ$, 
and in particular irreducible in $\bQ[X]$. 

In the following, we fix a prime $p$, and write $\bar{f}\in \bF_p[X]$ for the
reduction modulo $p$ of a polynomial $f\in \bZ[X]$. 

\begin{proposition}\label{prop:cyclo_modp_1}
Let $n$ be a positive integer, and write
$n = p^e m$ where $e,m\in \bZ_{\geq 0}$ and $\gcd(p,m) = 1$. 
Then $\overline{\Phi_n}(X) = \overline{\Phi_m}(X)^{\varphi(p^e)}$ in $\bF_p[X]$. 
\end{proposition}
\begin{proof}
For each $j\in \bZ$ with $0 \leq j \leq e$, the polynomial $\Phi_m(X^{p^e})$ vanishes 
at $X = \zeta_{p^j m}$ because the power $\zeta_{p^j m}^{p^e} = \zeta_m^{p^{e-j}}$ 
is a primitive $m$-th root of unity. This means that
$\Phi_m(X^{p^e})$ is divisible by $\Phi_m(X) \Phi_{pm}(X) \cdots \Phi_{p^e m}(X)$
since $\Phi_{p^jm}(X)$ is the minimal polynomial of $\zeta_{p^j m}$. 
On the other hand, we have 
\[ \begin{split}
&\deg(\Phi_m(X) \Phi_{pm}(X) \cdots \Phi_{p^e m}(X))
= \varphi(m) + \sum_{j=1}^e \varphi(p^{j}m) 
= \varphi(m) + \sum_{j=1}^e \varphi(p^{j})\varphi(m) \\
&\qquad= \varphi(m) + \sum_{j=1}^e \left(p^{j-1}(p-1)\varphi(m)\right) 
= \varphi(m) + (p-1)\varphi(m)\sum_{j=1}^e p^{j-1}  \\
&\qquad = \varphi(m) + \varphi(m)(p^e - 1) = \varphi(m)p^e 
= \deg(\Phi_m(X^{p^e})). 
\end{split}
\]
Thus 
$\Phi_m(X) \Phi_{pm}(X) \cdots \Phi_{p^e m}(X) = \Phi_m(X^{p^e})$, and 
\[
\overline{\Phi_m}(X) \overline{\Phi_{pm}}(X) \cdots \overline{\Phi_{p^e m}}(X) 
= \overline{\Phi_m}(X^{p^e}) = \overline{\Phi_m}(X)^{p^e}
\quad(\text{in $\bF_p[X]$}). \]
By induction on $e$, we obtain
$\overline{\Phi_{n}}(X) 
= \overline{\Phi_{p^e m}}(X) 
= \overline{\Phi_m}(X)^{\varphi(p^e)}
$ as required. 
\end{proof}

We then consider the factorization of $\overline{\Phi_m}(X)$ for 
a positive integer $m$ with $\gcd(p, m) = 1$. 
It is clear that $\overline{\Phi_1}(X) = X - 1$ and $\overline{\Phi_2}(X) = X + 1$. 
Let $m\geq 3$ be a positive integer with $\gcd(p, m) = 1$.  
We write $\overline{\bF_p}$ for the algebraic closure of $\bF_p$, and 
$\bF_q$ for the unique subfield of $\overline{\bF_p}$ whose cardinality  
equals $q$, that is, 
$\bF_q = \{ \alpha \in \overline{\bF_p} \mid \alpha^{q} = \alpha \}$, 
where $q$ is a power of $p$. 

\begin{proposition}\label{prop:cyclo_modp_2}
Let $\bar{f}\in \bF_p[X]$ be any irreducible factor of $\overline{\Phi_m}$,  
and let $d$ denote the order of $p$ in the multiplicative group 
$(\bZ/m\bZ)^\times$. 
\begin{enumerate}
\item Let $\bar{\zeta}\in \overline{\bF_p}$ be any root of $\bar{f}$. 
Then the order of $\bar{\zeta}$ in $\overline{\bF_p}^\times$ equals $m$. 
\item $\deg(\bar{f}) = d$. 
\item $\bar{f}$ is $+1$-symmetric if and only if 
there exists a non-negative integer $r$ such that $p^r \equiv -1 \bmod m$.
\end{enumerate}
In particular, all irreducible factors of $\overline{\Phi_m}$ have degree $d$ and 
are all $+1$-symmetric or all not. 
\end{proposition}
\begin{proof}
(i). Let $m'$ denote the order of $\bar{\zeta}$ in $\overline{\bF_p}^\times$. 
Since $\bar{f}(X)\mid \overline{\Phi_m}(X) \mid (X^m-1)$, 
we have $\bar{\zeta}^m = 1$, which shows that $m' \mid m$.
Note that $X^m - 1\in \bF_p[X]$ is separable since 
it is coprime to its derivative $m X^{m-1}$ in $\bF_p[X]$ 
by the assumption $\gcd(p,m) = 1$. 
If $m'$ were less than $m$ then the polynomial
\[X^m - 1 
= \prod_{r\mid m} \overline{\Phi_r}(X)
= (X^{m'} - 1)\prod_{r\mid m, r\nmid m'} \overline{\Phi_r}(X)
\]
would have $\bar{\zeta}$ as a multiple root since $\bar\zeta^{m'} - 1 = 0$ and 
$\overline{\Phi_m}(\bar\zeta) = 0$, but it contradicts separability. 
Therefore $m' = m$. 

(ii). Let $\bar{\zeta}\in \overline{\bF_p}$ be a root of $\bar{f}$. 
Since $p^d \equiv 1 \bmod m$ and $\bar{\zeta}^m = 1$, it follows that 
$\bar{\zeta}^{p^d} = \bar{\zeta}$. 
This shows that $\bar{\zeta} \in \bF_{p^d}$, and 
$\bF_p(\bar{\zeta}) \subset \bF_{p^d}$. Thus 
\[ \deg(\bar{f}) = [\bF_p(\bar{\zeta}): \bF_p] 
\leq [\bF_{p^d}: \bF_p] = d. \]
On the other hand, we have 
$\bar{\zeta}^{p^{\deg(\bar{f})} - 1} = 1$ since 
the extension $\bF_p(\bar{\zeta})/\bF_p$ is of degree $\deg(\bar{f})$. 
Hence $p^{\deg(\bar{f})} - 1\equiv 0 \bmod m$ by (i), 
and $p^{\deg(\bar{f})} \equiv 1 \bmod m$.  
Therefore $d\mid \deg(\bar{f})$, which shows that $d = \deg(\bar{f})$.  

(iii). Suppose that $\bar{f}$ is $+1$-symmetric. Then 
$\bar{f}(\bar{\zeta}^{-1}) = \bar{f}^*(\bar{\zeta}^{-1}) = 0$ for a root 
$\bar{\zeta}$ of $\bar{f}$. 
In other words, $\bar{\zeta}^{-1}$ and $\bar{\zeta}$ are conjugate. 
Thus, there exists $r\in\bZ_{\geq0}$ such that 
$\bar{\zeta}^{p^r} = \bar{\zeta}^{-1}$, or equivalently
$\bar{\zeta}^{p^r + 1} = 1$. 
This implies that $p^r + 1 \equiv 0 \bmod m$ by (i), and 
$p^r \equiv -1 \bmod m$.

Conversely, suppose that there exists $r\in\bZ_{\geq0}$ such that $p^r \equiv -1 \bmod m$. 
Then, for any root $\bar{\zeta}$ of $\bar{f}$, its inverse
$\bar{\zeta}^{-1} = \bar{\zeta}^{p^r}$ is conjugate to $\bar{\zeta}$. 
Noting that every root of $\bar{f}$ has multiplicity $1$ since $\bar{f}$ is separable, 
it follows that $\bar{f}$ is $*$-symmetric.
Furthermore, any root of $\bar{\zeta}$ is not $1$ since its order in 
$\overline{\bF_p}^\times$ is $m\geq 3$. 
Hence $X-1$ is not a factor of $\bar{f}$, and $\bar{f}$ is $+1$-symmetric. 
\end{proof}

The $n$-th cyclotomic polynomial decomposes in $\bF_p[X]$ as follows. 

\begin{theorem}\label{th:factorizationofCTPmodp}
Let $n$ be a positive integer, and write
$n = p^e m$ where $e,m\in \bZ_{\geq 0}$ and $\gcd(p,m) = 1$. 
\begin{enumerate}
\item $\overline{\Phi_n}(X) = \overline{\Phi_m}(X)^{\varphi(p^e)}$ in $\bF_p[X]$. 
\item Suppose that $m \geq 3$, and let $d$ denote the order of $p$ in 
$(\bZ/m\bZ)^\times$. 
Then $\overline{\Phi_m}$ has exactly $\varphi(m)/d$ irreducible factors of degree $d$ 
in $\bF_p[X]$. Moreover, the irreducible factors are all $+1$ symmetric or all not, and 
the former case occurs if and only if there exists $r \in \bZ_{\geq 0}$ such that 
$p^r \equiv -1 \bmod m$.  
\end{enumerate}
\end{theorem}
\begin{proof}
(i) is nothing but Proposition \ref{prop:cyclo_modp_1}, and 
(ii) follows from Proposition \ref{prop:cyclo_modp_2}. 
\end{proof}

\subsection{Computation of $\Pi(\Phi_n,\Phi_{n'})$}\label{ss:Comp_Pi_cyclos}
This subsection gives an explicit description of the set $\Pi(\Phi_n,\Phi_{n'})$. 
We remark that $\Phi_n$ is $+1$-symmetric and of even degree if $n \geq 3$. 
In this case, we write $\Psi_n \in \bZ[X]$ for the trace polynomial 
of $\Phi_n$ (see Definition \ref{def:tracepl}). 

\begin{lemma}\label{lem:thenumberof}
Let $n$ be a positive integer at least $3$, and $p$ a prime. 
The number of irreducible factors of $\Phi_n$ 
(resp. $\Psi_n$) over $\bQ_p$ is equal to that of $(\Phi_n \bmod p)$ 
(resp. $(\Psi_n \bmod p)$) over $\bF_p$. 
\end{lemma}
\begin{proof}
Let $l$ and $l'$ are the numbers of irreducible factors of $\Phi_n$
and that of $\Phi_n \bmod p$ respectively. 
Furthermore, let $l''$ be the number of places of the field $\bQ[X]/(\Phi_n)$ above $p$. 
It is known that the ring of integers of $\bQ[X]/(\Phi_n)$ is $\bZ[X]/(\Phi_n)$. 
Thus, we obtain $l = l''$ by \cite[Chapter I, Proposition 8.3]{Ne99}. 
On the other hand, it follows from \cite[Chapter II, Proposition 8.2]{Ne99} that 
$l' = l''$. Hence $l = l'$ as required. 
The assertion for $\Psi_n$ is also obtained in the same way. 
\end{proof}

\begin{proposition}\label{prop:I(Phi_n;p)}
Let $n\in \bZ_{>0}$ be a positive integer, and $p$ a prime. 
\begin{enumerate}
\item If $n = p^e$ for some $e\geq 0$ then $\overline{I(\Phi_{n};p)} = \{X-1\}$. 
\item If $n = 2p^e$ for some $e\geq 0$ then $\overline{I(\Phi_{n};p)} = \{X+1\}$. 
\item Suppose that $n = p^e m$ for some $e\geq 0$ and $m \geq 3$ with $\gcd(p,m) = 1$. 
Then 
\[ \overline{I(\Phi_{n};p)} =
\{ \bar{h}\in \bF_p[X] \mid \text{$\bar{h}$ is an irreducible factor of 
$(\Phi_m \bmod p)$} \} 
\]
if there exists $r\in\bZ_{\geq 0}$ such that $p^r \equiv -1 \bmod m$, and 
$\overline{I(\Phi_{n};p)} = \emptyset$ otherwise. 
\end{enumerate}
\end{proposition}
\begin{proof}
We only show the assertion (iii) because (i) and (ii) can be proved more easily 
in a similar way to the former case of (iii). 
Assume that $n = p^e m$ for some $e\geq 0$ and $m \geq 3$ with $\gcd(p,m) = 1$.  
Suppose first that there exists $r\in\bZ_{\geq 0}$ such that $p^r \equiv -1 \bmod m$. 
Then, Theorem \ref{th:factorizationofCTPmodp} implies that 
$\Phi_n \bmod p \in \bF_p[X]$ can be expressed as 
\[ \Phi_n \bmod p = \bar{f_1}^{\varphi(p^e)} \cdots \bar{f_l}^{\varphi(p^e)},  
\]
where $\bar{f_1}, \ldots, \bar{f_l}\in \bF_p[X]$ are distinct $+1$-symmetric 
irreducible factors of $\Phi_m \bmod p$. Our assertion is that 
$\overline{I(\Phi_n; p)} = \{ \bar{f}_1, \ldots, \bar{f}_l \}$. 
Let $\bar{h}_i\in \bF_p[X]$ be the trace polynomial of 
$\bar{f_i}$ ($i= 1, \ldots, l$). Then  
$\Psi_n \bmod p = \bar{h}_1^{\varphi(p^e)} \cdots \bar{h}_l^{\varphi(p^e)}$. 
Hensel's lemma shows that there exist 
$\hat{h}_1, \ldots, \hat{h}_l\in \bZ_p[X]$ such that  
$\Psi_n = \hat{h}_1 \cdots \hat{h}_l$ and 
$\hat{h}_i \bmod p = \bar{h}_i^{\varphi(p^e)}$ for any $i = 1, \ldots, l$. 
Furthermore $\hat{h}_1, \ldots, \hat{h}_l$ must be irreducible by 
Lemma \ref{lem:thenumberof}. 
Put $\hat{f}_i(X) = X^{\deg{\hat{h}_i}} \hat{h}_i(X + X^{-1})\in \bQ_p[X]$. 
Then $\hat{f}_i \bmod p  = \bar{f}_i^{\varphi(p^e)}$ for all $i$, and 
$\hat{f}_1, \ldots, \hat{f}_l$ are distinct because so are 
$\hat{h}_1, \ldots, \hat{h}_l$. 
Since $\Psi_n = \hat{h}_1 \cdots \hat{h}_l$, we have 
$\Phi_n = \hat{f}_1 \cdots \hat{f}_l$. 
Thus $\hat{f}_1, \ldots, \hat{f}_l$ are irreducible by Lemma \ref{lem:thenumberof}.  
Therefore
\[ \begin{split}
\overline{I(\Phi_{n};p)} 
&= \{ \bar{k}\in \bF_p[X] \mid \text{there exists $i$ such that
$\bar{k}$ is an irreducible factor of $\hat{f}_i \bmod p$}\} \\
&= \{ \bar{f}_1, \ldots, \bar{f}_l \}.
\end{split}
\]

Suppose then that there in no $r\in\bZ_{\geq 0}$ such that $p^r \equiv -1 \bmod m$. 
Then $\Phi_n \bmod p \in \bF_p[X]$ has no $*$-symmetric irreducible factor by Theorem 
\ref{th:factorizationofCTPmodp}, and so does $\Phi_n \in \bZ[X]$. 
In particular, we have $\overline{I(\Phi_n;p)} = \emptyset$. 
\end{proof}

We use the following formula.  

\begin{proposition}[Apostol]\label{prop:Apostol}
Let $n\in \bZ_{>0}$ be a positive integer. Then
\[ \Res(\Phi_n, \Phi_1) = \begin{cases}
-p &\text{if $n$ is a power of a prime $p$} \\ 
-1 &\text{otherwise.}
\end{cases}
\]
Let $n'\in \bZ_{>0}$ be a positive integer such that $1 < n' < n$. Then
\[ \Res(\Phi_n, \Phi_{n'}) = \begin{cases}
p^{\varphi(n')} &\text{if $n/n'$ is a power of a prime $p$} \\ 
1 &\text{otherwise.}
\end{cases}
\]
\end{proposition}
\begin{proof}
See \cite{Ap70}. 
\end{proof}

\begin{theorem}\label{th:Pi_cyclos}
Let $n$ and $n'$ be positive integers with $n>n'$. 
\begin{enumerate}
\item If $n/n'$ is not a power of a prime, then $\Pi(\Phi_n,\Phi_{n'}) = \emptyset$. 
\item Suppose that $n/n'$ is a power of a prime $p$, and write $n = p^e m$ where 
$\gcd(p, m) = 1$. Then 
\[
\Pi(\Phi_n,\Phi_{n'}) = \begin{cases}
\{ p \} & \text{if there exists $r\in\bZ_{\geq 0}$ such that $p^r \equiv -1 \bmod m$} \\
\emptyset & \text{otherwise}.
\end{cases}
\]
\end{enumerate} 
\end{theorem}
\begin{proof}
Let $n$ and $n'$ be positive integers with $n>n'$. 
If $n/n'$ is not a power of a prime then $|\Res(\Phi_n, \Phi_{n'})| = 1$ by 
Proposition \ref{prop:Apostol} and 
thus $\Pi(\Phi_n, \Phi_{n'}) = \emptyset$ by Proposition \ref{prop:Pi<Res}. 
This shows the assertion (i). 
We then show the assertion (ii). Suppose that $n/n'$ is a power of a prime $p$, 
and write $n = p^e m$, $n' = p^{e'} m$ where $e > e' > 1$ and $\gcd(p, m) = 1$. 
Note that $\Pi(\Phi_n, \Phi_{n'})\subset \{p\}$ by Propositions 
\ref{prop:Apostol} and \ref{prop:Pi<Res}. 

Suppose that $m = 1$ or $2$. Then 
$\overline{I(\Phi_n;p)}$ and $\overline{I(\Phi_{n'};p)}$ are both $\{X\mp 1\}$ by 
Proposition \ref{prop:I(Phi_n;p)} (i) or (ii), and $p \in \Pi(\Phi_n, \Phi_{n'})$. 
Thus $\Pi(\Phi_n, \Phi_{n'})= \{p\}$. 
Suppose that $m\geq 3$. It follows similarly from Proposition \ref{prop:I(Phi_n;p)}
(iii) that
\[
\Pi(\Phi_n,\Phi_{n'}) = \begin{cases}
\{ p \} & \text{if there exists $r\in\bZ_{\geq 0}$ such that $p^r \equiv -1 \bmod m$} \\
\emptyset & \text{otherwise}.
\end{cases}
\]
This completes the proof. 
\end{proof}

\subsection{Comparison}\label{ss:comparison}
One practical way to compute an obstruction map is to compare it with another. 
Let $F\in \bZ[X]$ be a $*$-symmetric polynomial of even degree with 
the condition \eqref{eq:Square}. 
As in Notation \ref{nt:LGPO}, the symbols $M$ and $M_\infty$ denote the associated 
$\bQ[X]$-module with transformation $\alpha$ and its localization at the infinite 
place $\infty$ respectively. 
Let $r,s$ be non-negative integers with $r\equiv s \bmod 8$ such that 
$F$ satisfies the condition \eqref{eq:Sign}${}_{r,s}$, and 
$\fraci\in \Idx(r,s;F)$ an index map. 
Then, there exists an inner product $b_\fraci$ on $M_\infty$
which makes $\alpha$ an isometry with index $\fraci$ (Theorem \ref{th:Sign_implies}). 
For any $f\in I := I(F;\bQ)$, the index of the restriction $b_\fraci|_{M^f_\infty}$ 
is given by $\sum_{g\in I(f;\bR)} \fraci(g)$, and thus its 
isomorphism class is uniquely determined by $\fraci$ independently of the choice of 
$b_\fraci$. So we define the map $\eta_\infty(\fraci): I \to \bZ/2\bZ$ by 
\[\eta_\infty(\fraci)(f) = \hw_\infty(b_\fraci|_{M^f_\infty}) \quad (f\in I).  \]
This map is explicitly as follows (see e.g. \cite[Chapter IV-\S 2.4]{Se73}): 
\begin{equation}\label{eq:eta_infty}
\eta_\infty(\fraci)(f) = \begin{cases}
0 & \text{if $(\deg(f^{m_f}) - \sum_{g\in I(f;\bR)} \fraci(g))/2 \equiv 0$ or $1$ mod $4$ }\\
1 & \text{if $(\deg(f^{m_f}) - \sum_{g\in I(f;\bR)} \fraci(g))/2 \equiv 2$ or $3$ mod $4$}
\end{cases}
\quad (f\in I).  
\end{equation}

Let $r',s'$ be non-negative integers with $r'\equiv s' \bmod 8$ such that 
$F$ satisfies the condition \eqref{eq:Sign}${}_{r',s'}$, and 
$\fj\in \Idx(r',s';F)$ an index map. We remark that if 
\begin{equation}\label{eq:idx_mod4}
 \fraci(X-1) \equiv \fj(X-1) \quad\text{and}\quad 
\fraci(X-1) \equiv \fj(X-1) \mod 4
\end{equation}
then two equivalence relations on $I$ defined by $(F, \fraci)$ and $(F, \fj)$ are 
the same, because they are determined by the values $\delta_+$ and $\delta_-$ defined in 
Notation \ref{nt:P1P2P3}, 
see also Notation \ref{nt:Pi_fg} and Definition \ref{def:ERdefinedby}. 
In this case, the obstruction group $\Omega_\fraci$ for $(F, \fraci)$ 
is the same as that for $(F, \fj)$. 

\begin{theorem}[Theorem \ref{th:comparison_intro}]\label{th:comparison}
Let $F\in \bZ[X]$ be a $*$-symmetric polynomial of even degree with 
the condition \eqref{eq:Square}, 
and $r,s, r', s'$ non-negative integers with $r\equiv s$ 
and $r'\equiv s' \bmod 8$ such that \eqref{eq:Sign}${}_{r,s}$ and 
\eqref{eq:Sign}${}_{r',s'}$ hold for $F$. 
Let $\fraci\in \Idx(r,s;F)$ and $\fj\in \Idx(r',s';F)$ be index maps 
satisfying \eqref{eq:idx_mod4}. Then we have 
\[ \ob_\fraci(c) = \ob_\fj(c) + (\eta_\infty(\fraci) - \eta_\infty(\fj))\cdot c \]
for all $c\in \Omega_\fraci$. In particular, 
if $\ob_\fj$ is zero, then $\ob_\fraci$ is zero precisely when the map
\[\Omega_\fraci \to \bZ/2\bZ,\, 
c\mapsto (\eta_\infty(\fraci) - \eta_\infty(\fj))\cdot c \]
is zero.
\end{theorem}
\begin{proof}
Let $\beta_\fraci = \{b_\fraci\}\cup \{ b_p \}_{p\in\cV\setminus\{\infty\}} 
\in \cB_\fraci$. Then the family 
$\beta_\fj := \{b_\fj\}\cup \{ b_p \}_{p\in\cV\setminus\{\infty\}}$ belongs to 
$\cB_\fj$, where $b_\fj$ is an inner product on $M_\infty$ which makes $\alpha$ an 
isometry with index $\fj$. For any $f\in I$ we have 
\[ \begin{split}
\eta(\beta_\fraci)(f) 
&= \hw_\infty(b_\fraci|_{M^f_\infty}) 
 + \sum_{p\in \cV\setminus\{\infty\}} \hw_p(b_p|_{M^f_p}) \\
&= \left(\hw_\infty(b_\fj|_{M^f_\infty}) 
 + \sum_{p\in \cV\setminus\{\infty\}} \hw_p(b_p|_{M^f_p}) \right)
 + \hw_\infty(b_\fraci|_{M^f_\infty}) - \hw_\infty(b_\fj|_{M^f_\infty}) \\
& = \eta(\beta_\fj)(f) + (\eta_\infty(\fraci) - \eta_\infty(\fj))(f).
\end{split}
\]  
This implies that 
$ \ob_\fraci(c) = \ob_\fj(c) + (\eta_\infty(\fraci) - \eta_\infty(\fj))\cdot c $
for all $c\in \Omega_\fraci$ as required. 
\end{proof}

\begin{example}
Let $F(X) = (X-1)^4f(X)$, where $f(X) = \Phi_{12}(X) = X^4 - X^2 + 1$ 
as in Example \ref{ex:Pi_cyclo12_X-1}. 
Note that the obstruction group $\Omega$ for $(F, \fraci)$ does not depend on 
the index map $\fraci$, 
and is given by $\Omega = C(I)$ since $\Pi(f, X-1) = \emptyset$ as 
shown in Example \ref{ex:Pi_cyclo12_X-1}, where $I = I(F; \bQ)$. 
Let $\Lambda$ be the lattice $\rE_8$. Then there is no isometry of 
$\Lambda$ with characteristic polynomial $F$ (any isometry of $\Lambda$ is semisimple 
because $\Lambda$ has definite signature). To prove this, 
we first show that $\Lambda_{4,4}$ admits a semisimple isometry with characteristic $F$, 
where $\Lambda_{n,n}$ is the unique even unimodular lattice of signature $(n,n)$ 
for $n\in \bZ_{>0}$. 
It follows from Theorem \ref{th:weakestER} that 
$\Lambda_{2,2}$ admits a semisimple isometry $t_f$ with characteristic 
polynomial $f$ because $f$ satisfies the conditions 
\eqref{eq:Sign}${}_{2,2}$ and \eqref{eq:Square}, and is irreducible. 
Thus the direct sum $t' := t_f \oplus \id_{\Lambda_{2,2}}$ is a semisimple isometry 
of $\Lambda_{4,4} = \Lambda_{2,2} \oplus \Lambda_{2,2}$ with characteristic polynomial $F$. 
Hence $\ob_\fj = 0$, where $\fj \in \Idx(4,4; F)$ denotes the index of $t'$.

We now suppose to contrary that the $\rE_8$-lattice $\Lambda$ admitted an isometry $t$ 
with characteristic polynomial $F$. 
Then the index $\fraci\in \Idx(8,0; F)$ of $t$ is uniquely determined: 
$\fraci(g) = \deg(g)^{m_g}$ for every $g \in I(F;\bR)$ where $m_g$ is the multiplicity 
of $g$ in $F$. Thus we would get 
\[ (\eta_\infty(\fraci) - \eta_\infty(\fj))\cdot\bm1_{\{X-1\}} 
= \eta_\infty(\fraci)(X-1) - \eta_\infty(\fj)(X-1)
= 0 - 1 
= 1 \quad \text{(in $\bZ/2\bZ$)},  \]
but this contradicts Theorem \ref{th:comparison}.
Therefore, there is no isometry of $\Lambda$ with characteristic polynomial $F$.
\end{example}

A similar idea of this example can be found in \cite[Example 3]{Ba21}. 
The method of computation as in this example is valid for most of polynomials 
thanks to Theorem \ref{th:chpl_on_index0_intro}, which is proved in the next section.

\section{Isometries on even unimodular lattices of index $0$}\label{sec:Iso_on_EULofIndex0}
This section is devoted to the proof of Theorem \ref{th:chpl_on_index0_intro}. 
For a positive integer $n$, we write $\Lambda_{n,n}$ for an even unimodular 
lattice over $\bZ$ of signature $(n,n)$. Such a lattice is unique up to 
isomorphism, see Theorem \ref{th:EULoverZ}. 

\subsection{General case}
Let $F\in \bR[X]$ be a $*$-symmetric polynomial. 
We write $m_f$ for the multiplicity of a polynomial 
$f$ in $F$ and $m_\pm := m_{X\mp 1}$. Furthermore $F_{12}$ is the product of 
type $1$ and $2$ components of $F$, and $m(F)$ is the number of roots of $F$ 
whose absolute values are greater than $1$ counted with multiplicity.

Let $P\in \bR[X]$ be a $*$-symmetric polynomial with $P(1)P(-1) \neq 0$. Then 
$P$ is $+1$-symmetric and of even degree. The signature of 
$(-1)^{\deg(P)/2}P(1)P(-1)$ is denoted by $e(P)\in \{1, -1\}$. 
If $H$ is the trace polynomial of $P$ then  $(-1)^{\deg(P)/2}P(1)P(-1) = H(2)H(-2)$, 
and thus $e(P)$ is the signature of $H(2)H(-2)$. 
Note that if the coefficients of $P$ are in $\bZ$ and $P$ satisfies the 
condition \eqref{eq:Square} then $e(P) = 1$ since $(-1)^{\deg(P)/2}P(1)P(-1)$ is 
a square. Let $\bT$ denote the unit circle in $\bC$. 

\begin{lemma}\label{lem:1-e_mod4}
The number of roots of a $*$-symmetric polynomial $F\in \bR[X]$ on $\bT\setminus\{1,-1\}$ 
counted with multiplicity is equal to $2\sum_{f\in I_1(F;\bR)} m_f$. Moreover,  
we have $2\sum_{f\in I_1(F;\bR)} m_f \equiv 1 - e(F_{12}) \bmod 4$.  
\end{lemma}
\begin{proof}
Let $N$ be the number of roots of $F$ on $\bT\setminus\{1,-1\}$ 
counted with multiplicity. 
The former assertion, $N = 2\sum_{f\in I_1(F;\bR)} m_f$, follows from 
the fact that any $+1$-symmetric irreducible polynomial over $\bR$ other than $X+1$ 
is of the form $X^2 - (\delta + \delta^{-1})X + 1$ for some complex number 
$\delta\in \bT\setminus\{1,-1\}$, and vice versa. 
Let $H$ denote the trace polynomial of $F_{12}$, and $N'$ the number of roots of $H$ 
on the interval $(-2,2)$. Then $N = 2N'$
because $\bT\setminus\{1,-1\}$ is mapped two-to-one onto $(-2,2)$
under the function $\bC \to \bC,\,x\mapsto x+x^{-1}$. 
On the other hand, by considering the graph of $H$, it can be seen that 
$N'$ is even if $H(2)H(-2) > 0$ and odd if $H(2)H(-2) < 0$. 
Since $e(F_{12})$ is the signature of $H(2)H(-2)$, we have 
$N' \equiv (1 - e(F_{12}))/2 \bmod 2$. Hence 
\[ N = 2 N' \equiv 1 - e(F_{12}) \mod 4. 
\]
The proof is complete. 
\end{proof}

\begin{proposition}\label{prop:index_prolong}
Let $F\in \bZ[X]$ be a $*$-symmetric polynomial of even degree $2n$ with the 
condition \eqref{eq:Square}. 
\begin{enumerate}
\item Let $r,s\in \bZ_{\geq 0}$ be non-negative integers with $r \equiv s \bmod 8$
and $r + s = 2n$ such that $F$ satisfies the condition \eqref{eq:Sign}${}_{r,s}$. Then 
we have $\fraci(X-1) + \fraci(X+1) \equiv 1 - e(F_{12}) \mod 4$
for any $\fraci \in \Idx(r,s;F)$.  
\item $F$ satisfies the condition \eqref{eq:Sign}$_{n,n}$. Moreover, 
for $i_+, i_-\in \bZ$ with 
\begin{equation}\label{eq:prolong}
\begin{split}
&-m_+ \leq i_+ \leq m_+,\, -m_- \leq i_- \leq m_-,\, 
i_+ \equiv i_- \equiv m_+ \bmod 2, \quad\text{and}\quad \\
&i_+ + i_- \equiv 1 - e(F_{12}) \bmod 4,  
\end{split}
\end{equation}
there exists $\fraci \in \Idx(n,n;F)$ such that 
$\fraci(X-1) \equiv i_+$ and $\fraci(X+1) \equiv i_-$ \textup{mod} $4$. 
\end{enumerate}
\end{proposition}
\begin{proof}
(i). Let $\fraci \in \Idx(r,s;F)$. We have 
\[ \textstyle
\fraci(X-1) + \fraci(X+1) + \sum_{f\in I_1(F;\bR)}\fraci(f) = r-s \equiv 0 \mod 4 
\] 
by \eqref{eq:indexdata3}. On the other hand, each $f\in I_1(F;\bR)$ satisfies 
$2m_f - \fraci(f) \equiv 0 \bmod 4$ by \eqref{eq:indexdata2}. Thus 
\[ \begin{split}
&\textstyle\fraci(X-1) + \fraci(X+1) 
\equiv -\sum_{f\in I_1(F;\bR)} \fraci(f)
\equiv -\sum_{f\in I_1(F;\bR)} 2 m_f \\
&\textstyle \qquad \equiv 2 \sum_{f\in I_1(F;\bR)} m_f
\equiv  1 - e(F_{12}) \mod 4, 
\end{split}
\] 
where the last congruence is by Lemma \ref{lem:1-e_mod4}. 

(ii). By Corollary \ref{cor:conditionSign}, it suffices to prove the 
latter assertion. Let $i_+, i_-\in \bZ$ be integers satisfying \eqref{eq:prolong}, 
and take $i_+'\in \{-1,0,1,2\}$ and $i_-'\in \{-2,-1,0,1\}$ so that  
$i_+' \equiv i_+$ and $i_-' \equiv i_-$ mod $4$. 
Then $-3 \leq i_+' + i_-' \leq 3$. Moreover, since 
$i_+' + i_-' \equiv i_+ + i_- \equiv 1 - e(F_{12}) \bmod 4$ we have 
\begin{equation}\label{eq*:i_+'+i_-'}
i_+' + i_-' = 1 - e(F_{12}) \text{ or $e(F_{12}) - 1$}. \tag{$*$}
\end{equation}
We now put 
\begin{alignat*}{2}
r_0 &= (m_+ + i_+')/2 + (m_- + i_-')/2, \quad
&s_0 &= (m_+ - i_+')/2 + (m_- - i_-')/2, \\
r_{12} &= n - r_0, 
& s_{12} &= n - s_0. 
\end{alignat*}
Then the map $\fraci_0: \{X-1, X+1\}\to \bZ$ defined by 
$\fraci_0(X-1) = i_+'$ and $\fraci_0(X+1) = i_-'$
belongs to $\Idx(r_0,s_0; F_0)$, where $F_0$ is the type $0$ component of $F$
(to be precise, $\fraci_0|_{I_0(F_0; \bR)} \in \Idx(r_0,s_0; F_0)$). 

We show that $F_{12}$ satisfies the condition \eqref{eq:Sign}$_{r_{12}, s_{12}}$. 
Note that $m(F) = (2n - m_+ - m_- - 2\sum_{f\in I_1(F;\bR)} m_f )/2$. Then 
\begin{equation}\label{eq*:r-m}
 r_{12} - m(F_{12}) 
= r_{12} - m(F) 
= \sum_{f\in I_1(F;\bR)} m_f - (i_+' + i_-')/2. \tag{$**$}
\end{equation}
Here $\sum_{f\in I_1(F;\bR)} m_f \geq (1 - e(F_{12}))/2$
since $\sum_{f\in I_1(F;\bR)} m_f \geq 0$ and 
$\sum_{f\in I_1(F;\bR)} m_f \equiv (1 - e(F_{12}))/2 \bmod 2$ by Lemma 
\ref{lem:1-e_mod4}. Hence, it follows from 
Equations \eqref{eq*:i_+'+i_-'} and \eqref{eq*:r-m} that 
\[\begin{split}
&r_{12} - m(F_{12}) \geq (1 - e(F_{12}))/2 - (1 - e(F_{12}))/2 = 0, \\ 
&r_{12} - m(F_{12}) \equiv (1 - e(F_{12}))/2 - (1 - e(F_{12}))/2 \equiv 0 \mod 2. 
\end{split}\]
Similarly, we get 
$s_{12} - m(F_{12}) \geq 0$ and $s_{12} - m(F_{12}) \equiv 0 \bmod 2$,  
and therefore $F_{12}$ satisfies the condition \eqref{eq:Sign}$_{r_{12}, s_{12}}$. 

Let $\fraci_1 \in \Idx(r_{12}, s_{12}; F_{12})$ be an index map. Then 
the sum $\fraci := \fraci_0\oplus \fraci_1: I(F;\bR)\to \bZ$ belongs to $\Idx(n,n;F)$. 
This is the desired index map since  
$\fraci(X\mp 1) = i_\pm' \equiv i_\pm \bmod 4$. 
The proof is complete. 
\end{proof}

We reformulate Notation \ref{nt:Pi_fg} and Definition \ref{def:ERdefinedby}. 

\begin{definition}\label{def:Pi_fg_reformulation}
Let $F\in \bZ[X]$ be a $*$-symmetric polynomial of even degree, and 
$m_\pm$ the multiplicity of $X\mp1$ in $F$. Let 
$i_+, i_-\in \bZ$ be integers with $i_\pm \equiv m_\pm \bmod 2$. 
We define 
\[\delta_\pm(F;i_+, i_-) := \begin{cases}
(-1)^{(m_\pm - i_\pm)/2}\,|F_{12}(\pm1)| & \text{if $m_+$ is even} \\
(-1)^{(m_\pm - i_\pm)/2}\,2|F_{12}(\pm1)| & \text{if $m_+$ is odd.}  
\end{cases}\]
Moreover
\[ \begin{split}
\overline{I(X\mp 1;\bQ_p)}' 
:= \begin{cases}
\{ X\mp 1 \} & \text{if $m_\pm \geq 3$; or $m_\pm = 2$ and 
$\delta_\pm(F;i_+, i_-) \neq -1 \in \bQ_p^\times/\bQ_p^{\times 2}$} \\
\emptyset & \text{otherwise}
\end{cases}
\end{split}\]
and $\overline{I(f;\bQ_p)}' := \overline{I(f;\bQ_p)}$ for a monic polynomial 
with $f(1)f(-1)\neq 0$ as in Notation \ref{nt:Pi_fg}. 
We then define a set $\Pi_{i_+, i_-}^F(f,g)$ of primes for monic polynomials 
$f,g\in \bZ[X]$ by 
\[ \Pi_{i_+, i_-}^F(f,g) := \{ p:\text{prime} 
\mid \overline{I(f;\bQ_p)}'\cap \overline{I(g;\bQ_p)}' \neq \emptyset \}. \]
The \textit{equivalence relation defined by $(F;i_+, i_-)$} is the one on $I(F;\bQ)$ 
generated by the binary relation
$\{ (f,g) \in I(F;\bQ) \times I(F;\bQ) \mid \Pi_{i_+, i_-}^F(f,g)\neq \emptyset \}$. 
Note that $\Pi_{i_+, i_-}^F(f,g) = \Pi_\fraci^F(f,g)$
if $\fraci$ is an index map with 
$\fraci(X-1) \equiv i_+$ and $\fraci(X+1) \equiv i_-$ mod $4$. 
\end{definition}

In the following, $F\in \bZ[X]$ is a $*$-symmetric polynomial of even degree $2n$
with the condition \eqref{eq:Square}, and $i_+,i_-$ are integers 
satisfying \eqref{eq:prolong}.  
Note that $m_+ + m_- = \deg(F_0)$ is even since so are $\deg(F)$ and $\deg(F_{12})$. 
Note also that the condition \eqref{eq:Square} holds for the factor 
$(X-1)^{m_+} (X+1)^{m_-} F_1(X)$ by Lemma \ref{lem:properties_for_Square}. 
For a subset $J$ in $I := I(F;\bQ)$, we define 
$F_J := \prod_{f\in J} f^{m_f}$, where $m_f$ is the multiplicity of $f$ in $F$.  
If $J$ is empty then $F_J$ is defined to be the constant $1$. 
The subset $J\setminus\{X-1, X+1\}$ of $J$ is denoted by $J^\circ$. 

\begin{lemma}\label{lem:FJ_of_pm1_is_zero}
Suppose that $m_+\neq 1$ and $m_-\neq 1$. Let $J$ be an equivalence class in 
$I$ with respect to the equivalence relation defined by $(F; i_+, i_-)$. 
If $|F_{J^\circ}(\pm1)|$ is not a square then $X\mp1 \in J$ i.e., $F_J(\pm1) = 0$. 
\end{lemma}
\begin{proof}
Suppose that $|F_{J^\circ}(\pm1)|$ is not a square, and let $p$ be a prime 
such that $v_p(F_{J^\circ}(\pm1))$ is odd. We remark that there is no 
$g\in I_1\setminus J$ such that $v_p(g(\pm1))$ is odd; because 
otherwise Proposition \ref{prop:nonSquare} implies that
$p \in \Pi_{i_+,i_-}^F(F_{J^\circ}, g)$, and we would have $g\in J$. 
Hence 
\[ v_p(F_{12}(\pm1)) \equiv v_p(F_1(\pm1)) \equiv v_p(F_{J^\circ}(\pm1)) \equiv 1 \mod 2. \]
This implies that $|F(\pm1)| = 0$ since $|F(\pm1)|$ is a square, 
and thus $m_\pm \geq 2$ by the assumption $m_\pm \neq 1$. Moreover, 
if $m_\pm = 2$ then $\delta_\pm(F;i_+, i_-) \neq -1$ in $\bQ_p^\times/\bQ_p^{\times2}$ 
since $v_p(\delta_\pm(F;i_+, i_-)) \equiv v_p(F_{12}(\pm1)) \equiv 1 \bmod 2$.  
Therefore $\overline{I(X\mp1;\bQ_p)}' = \{X\mp1\}$, 
which leads to $p\in \Pi_{i_+,i_-}^F(F_{J^\circ}, X\mp1)$, and $X\mp1 \in J$. 
This means that $F_J(\pm1) = 0$, and the proof is complete. 
\end{proof}

\begin{proposition}\label{prop:EC_satisfies_Square}
Suppose that $m_+\neq 1$ and $m_-\neq 1$. 
For any equivalence class $J$ in $I$ with respect to the equivalence relation 
defined by $(F; i_+, i_-)$, the corresponding factor 
$F_J$ has even degree and satisfies the condition \eqref{eq:Square}.  
\end{proposition}
\begin{proof}
We begin with the case $F(1)F(-1) \neq 0$, i.e., $m_+ = m_- = 0$. 
Let $J$ be an equivalence class in $I$. 
The degree of $F_J$ is even since $F_J$ has no type $0$ component. 
We first show that $|F_J(1)|$ and $|F_J(-1)|$ are squares simultaneously. 
Suppose that $|F_J(\pm1)|$ were not a square, and let $p$ be prime such that 
$v_p(F_J(\pm1))$ is odd. Note that $v_p(F_1(\pm1))$ is even since 
the condition \eqref{eq:Square} holds for $F_1$. 
Then there exists $g\in I\setminus J$ such that $v_p(g(\pm1))$ is odd. 
In this case, we have $X\mp1 \in \overline{I(F_J;\bQ_p)} = \overline{I(F_J;\bQ_p)}'$ 
and $X\mp1 \in \overline{I(g;\bQ_p)} = \overline{I(g;\bQ_p)}'$ by 
Proposition \ref{prop:nonSquare}. 
Thus $p \in \Pi_{i_+,i_-}^F(F_J,g)$, and 
$g$ would be contained in the equivalence class $J$, but this is a contradiction.  
Hence $|F_J(1)|$ and $|F_J(-1)|$ are squares. 

We then show that $(-1)^{n_J}F_J(1)F_J(-1)$ is a square, where $n_J := \deg(F_J)/2$. 
Suppose that $(-1)^{n_J}F_J(1)F_J(-1)$ were not a square. 
Since $|F_J(1)|$ and $|F_J(-1)|$ are squares as proved now, 
we have $(-1)^{n_J}F_J(1)F_J(-1) = -1$ 
in $\bQ^\times/\bQ^{\times 2}$ and hence in $\bQ_2^\times/\bQ_2^{\times 2}$. 
Thus there exists $g\in I\setminus J$ such that 
$(-1)^{\deg(g)/2}g(1)g(-1) \neq 1$ nor $-3$ in $\bQ_2^\times/\bQ_2^{\times 2}$
because $(-1)^{\deg(F_1)/2}F_1(1)F_1(-1) = 1$ in $\bQ_2^\times/\bQ_2^{\times 2}$. 
In this case, we have $X-1 \in \overline{I(F_J;\bQ_p)} = \overline{I(F_J;\bQ_p)}'$ 
and $X-1 \in \overline{I(g;\bQ_p)} = \overline{I(g;\bQ_p)}'$ by 
Proposition \ref{prop:nonSquare}. Thus $2 \in \Pi_{i_+,i_-}^F(F_J,g)$, and 
$g$ would be contained in the equivalence class $J$, but this is a contradiction.  
Hence $(-1)^{n_J}F_J(1)F_J(-1)$ is a square. 
This completes the case $F(1)F(-1) \neq 0$.

We proceed to the case $F(1)F(-1) = 0$. Let $J\subset I$ be an equivalence class. 
Note that $F_{J^\circ}$ has even degree. 
If $m_+$ and $m_-$ are even then it is obvious that $\deg(F_J)$ is even. 
If $m_+$ and $m_-$ are odd then $m_+, m_- \geq 3$ by the assumption $m_+, m_- \neq 1$. 
This implies that $J = J^\circ$ or $\{X-1,X+1\}\subset J$, and 
$\deg(F_J)$ is even. 
We now show that $F_J$ satisfies the condition \eqref{eq:Square}. 
If $F_{J^\circ}$ satisfies \eqref{eq:Square} then so does $F_J$. 
Hence, it is enough to consider the case where $F_{J^\circ}$ does not satisfy 
\eqref{eq:Square}.
Lemma \ref{lem:FJ_of_pm1_is_zero} implies that if 
$|F_{J^\circ}(1)|$ or $|F_{J^\circ}(-1)|$ is not a square then 
$F_J$ satisfies \eqref{eq:Square}. 
Suppose then that $|F_{J^\circ}(1)|$ and $|F_{J^\circ}(-1)|$ are squares but 
$(-1)^{\deg(F_{J^\circ})/2}F_{J^\circ}(1)F_{J^\circ}(-1) = -1$ mod squares.  
If $m_+\geq 3$ or $m_-\geq 3$ then $\overline{I(X-1; \bQ_2)}' = \{X-1\}$ or 
$\overline{I(X+1; \bQ_2)}' = \{X-1\}$, and 
$2\in \Pi_{i_+,i_-}^F(F_J, X-1)$ or $2\in \Pi_{i_+,i_-}^F(F_J, X+1)$ 
by Proposition \ref{prop:nonSquare}. 
Thus $X-1\in J$ or $X+1\in J$, and $F_J(1)F_J(-1) = 0$. This leads to 
the condition \eqref{eq:Square} for $F_J$. Suppose that 
$(m_+, m_-) = (2,2)$, $(2,0)$ or $(0,2)$. 
Note that there is no $g\in I_1\setminus J$ such that 
$(-1)^{\deg(g)/2}g(1)g(-1) \neq 1$ nor $-3$ since otherwise we would have $g\in J$ 
by Proposition \ref{prop:nonSquare}. 
This shows that in $\bQ_2^\times/\bQ_2^{\times2}$ we have
\[ \begin{split}
&(-1)^{\deg(F_{12})/2} F_{12}(1) F_{12}(-1) \\
&= (-1)^{\deg(F_{J^\circ})/2} F_{J^\circ}(1) F_{J^\circ}(-1)
\times \prod_{g\in I_1\setminus J} (-1)^{\deg(g)/2}g(1)g(-1) 
\times (-1)^{\deg(F_2)/2} F_{2}(1) F_2(-1) \\ 
&= - \prod_{g\in I_1\setminus J} (-1)^{\deg(g)/2}g(1)g(-1) \\
&\in \{-1, 3\}.  
\end{split} \]
Thus 
\[\begin{split}
\delta_+(F; i_+, i_-)\delta_-(F; i_+, i_-)
&= (-1)^{(m_+ - i_+)/2} |F_{12}(1)| (-1)^{(m_+ - i_+)/2} |F_{12}(-1)| \\
&= (-1)^n F_{12}(1)F_{12}(-1) \\
&= (-1)^{(m_+ + m_-)/2} (-1)^{\deg(F_{12})/2} F_{12}(1)F_{12}(-1) \\
&= \begin{cases}
1 \text{ or } -3 & \text{if $(m_+, m_-) = (2,0)$ or $(0,2)$} \\
-1 \text{ or } 3 & \text{if $(m_+, m_-) = (2,2)$}
\end{cases}
\end{split}\]
in $\bQ_2^\times/\bQ_2^{\times2}$, where the second equality is obtained by applying 
Lemma \ref{lem:signature_of_F12(1)F12(-1)} 
after taking $\fraci$ as in Proposition \ref{prop:index_prolong} (ii). 
Hence, if $(m_+, m_-) = (2,0)$ then $\delta_+(F; i_+, i_-) \neq -1$;
if $(m_+, m_-) = (0,2)$ then $\delta_-(F; i_+, i_-) \neq -1$; and 
if $(m_+, m_-) = (2,2)$ then $\delta_-(F; i_+, i_-) \neq -1$ or 
$\delta_-(F; i_+, i_-) \neq -1$ in $\bQ_2^\times/\bQ_2^{\times2}$. 
These imply that $X-1\in \overline{I(X-1;\bQ_2)}'$ or 
$X-1\in \overline{I(X+1;\bQ_2)}'$, and $X-1\in J$ or $X+1 \in J$. 
Therefore $(-1)^{\deg(F_J)/2}F_J(1) F_J(-1) = 0$, which leads to the condition 
\eqref{eq:Square} for $F_J$. The proof is complete. 
\end{proof}

Let $J_\pm$ denote the equivalence class in $I$ (with respect to the 
equivalence relation defined by $(F; i_+, i_-)$) containing $X\mp1$. 
If $m_\pm = 0$ then $J_\pm$ is defined to be the empty set. 
Note that $J_+$ and $J_-$ coincide or have no intersection. 

\begin{lemma}\label{lem:ECweakest}
Suppose that $m_+\neq 1$ and $m_-\neq 1$. Then $\delta_+(F_{J_+};i_+, i_-) 
= \delta_+(F;i_+, i_-)$ and  
$\delta_-(F_{J_-};i_+, i_-) = \delta_-(F;i_+, i_-)$ in $\bQ^\times/ \bQ^{\times 2}$. 
As a result, 
for any equivalence class $J$ in $I$ with respect to the equivalence relation 
defined by $(F; i_+, i_-)$, the equivalence relation on $J = I(F_J; \bQ)$ 
defined by $(F_J; i_+, i_-)$ is weakest. 
\end{lemma}
\begin{proof}
Suppose first that $J_+ \neq J_-$. Then 
$F_{12} = F_{J_+^\circ} F_{J_-^\circ} \times \prod_{H\neq J_+,J_-}F_H \times F_2$, 
where $H$ ranges over all equivalence classes other than $J_+$ and $J_-$. 
Note that $|F_{J_-^\circ}(1)|$ must be a square since otherwise 
$X-1$ would belong to $J_-$ by Lemma \ref{lem:FJ_of_pm1_is_zero}, but this 
contradicts $J_+ \neq J_-$. Furthermore 
$|F_H(1)|$ for $H \neq J_+, J_-$ and $|F_2(1)|$ are also squares 
by Proposition \ref{prop:EC_satisfies_Square} and 
Lemma \ref{lem:properties_for_Square} (ii). Thus, we obtain 
\[ |F_{12}(1)| 
= |F_{J_+^\circ}(1)| |F_{J_-^\circ}(1)| \times \prod_{H\neq J_+,J_-}|F_H(1)| \times |F_2(1)|
= |F_{J_+^\circ}(1)|
\quad \text{in $\bQ^\times/ \bQ^{\times 2}$}.  
\] 
If $J_+ = J_-$ then 
$F_{12} = F_{J_+^\circ}\times \prod_{H\neq J_+}F_H \times F_2$, and 
it also follows from Proposition \ref{prop:EC_satisfies_Square} and 
Lemma \ref{lem:properties_for_Square} (ii) that 
$|F_{12}(1)| =  |F_{J_+^\circ}(1)|$ in $\bQ^\times/ \bQ^{\times 2}$.
Furthermore, the multiplicity of $X-1$ in $F_{J_+}$ is $m_+$. Hence 
$\delta_+(F_{J_+};i_+, i_-) 
= \delta_+(F;i_+, i_-)$ in $\bQ^\times/ \bQ^{\times 2}$. 
Similarly, we have $\delta_-(F_{J_-};i_+, i_-) = \delta_-(F;i_+, i_-)$
in $\bQ^\times/ \bQ^{\times 2}$. 

For the latter assertion, 
it suffices to check that $\Pi_{i_+, i_-}^{F_J}(f,g) = \Pi_{i_+, i_-}^{F}(f,g)$
for any $f,g \in J$. If $f,g \notin \{X-1, X+1\}$ then 
$\Pi_{i_+, i_-}^{F_J}(f,g)$ and $\Pi_{i_+, i_-}^{F}(f,g)$ are equal to 
$\Pi(f,g)$ (defined in Notation \ref{nt:Pi_fg_pure}), and they coincide. 
So it is enough to prove this assertion for $J_+$ and $J_-$, but it 
follows from the equalities $\delta_+(F_{J_+};i_+, i_-) = \delta_+(F;i_+, i_-)$ and
$\delta_-(F_{J_-};i_+, i_-) = \delta_-(F;i_+, i_-)$ 
in $\bQ^\times/ \bQ^{\times 2}$, which we proved above. 
\end{proof}

The following lemma is a special case of the main theorem \ref{th:chpl_on_index0}.  

\begin{lemma}\label{lem:if_J+=J-}
Suppose that $m_+ \neq 1$ and $m_- \neq 1$. 
Let $i_+, i_-\in \bZ$ be integers with \eqref{eq:prolong}. 
If $J_+\cup J_-$ is one equivalence class in $I$ with respect to 
the equivalence relation defined by $(F; i_+, i_-)$
then there exists $\fraci \in \Idx(n,n;F)$ with 
$\fraci(X-1) \equiv i_+$ and $\fraci(X+1) \equiv i_-$ mod $4$
such that $\Lambda_{n,n}$ admits a semisimple $(F,\fraci)$-isometry. 
\end{lemma}
\begin{proof}
Let $H$ be an equivalence class in $I$ other than $J_+$ and $J_-$, 
and put $n_H = \deg(F_H)/2$. 
Then $F_H$ satisfies the conditions 
\eqref{eq:Sign}$_{n_H,n_H}$ and \eqref{eq:Square} 
by Proposition \ref{prop:index_prolong} (ii) and 
Proposition \ref{prop:EC_satisfies_Square}. 
Moreover, the equivalence relation on $H = I(F_H; \bQ)$ 
defined by $(F_H; i_+, i_-)$ is weakest by Lemma \ref{lem:ECweakest}. 
Let $\fraci_H\in \Idx(n_H,n_H; F_H)$ be an index map. 
Theorem \ref{th:weakestER} shows that
$\Lambda_{n_H, n_H}$ has a semisimple $(F_H, \fraci_H)$-isometry $t_H$. 

Suppose that $J := J_+\cup J_-$ is one equivalence class in $I$. 
Then $F_1 = F_J \times \prod_H F_H$, where $H$ ranges over all equivalence classes 
other than $J$. Note that $e(F_H) = 1$ for each $H\neq J$ and $e(F_2) = 1$
since $F_H$ and $F_2$ satisfy the condition \eqref{eq:Square}. Thus
\begin{equation}\label{eq:e(F_12)}
e(F_{12}) = e(F_J)\times \prod_{H\neq J} e(F_H) \times e(F_2) = e(F_J).
\end{equation}
Furthermore, the multiplicities of $X-1$ and $X+1$ in $F_J$ are $m_+$ and $m_-$ 
respectively. Hence, there exists $\fraci_J \in \Idx(n_J,n_J;F_J)$ with 
$\fraci_J(X-1) \equiv i_+$ and $\fraci_J(X+1) \equiv i_-$ mod $4$
by Proposition \ref{prop:index_prolong} (ii), where $n_J := \deg(F_J)/2$. 
Lemma \ref{lem:ECweakest} and Theorem \ref{th:weakestER} imply that 
$\Lambda_{n_J, n_J}$ has a semisimple $(F_J, \fraci_J)$-isometry $t_J$. 

We now define an isometry $t$ on 
$\Lambda_{n,n} = \Lambda_{n_J, n_J} \oplus \bigoplus_{H\neq J} \Lambda_{n_H, n_H}$
by $t := t_J\oplus \bigoplus_{H\neq J} t_H$, and define $\fraci\in \Idx(n,n;F)$ 
to be the index of $t$. Then 
$\fraci(X-1) \equiv \fraci_J(X-1) \equiv i_+$ and 
$\fraci(X+1) \equiv \fraci_J(X+1) \equiv i_-$ mod $4$. 
This completes the proof. 
\end{proof}

The following proposition is an essential part of the proof of the main theorem.  

\begin{proposition}\label{prop:if_m+=m-=2}
Let $i_+, i_-\in \bZ$ be integers with \eqref{eq:prolong}. 
Suppose that $m_+ = m_- = 2$. Then there exists $\fraci \in \Idx(n,n;F)$ with 
$\fraci(X-1) \equiv i_+$ and $\fraci(X+1) \equiv i_-$ mod $4$
such that $\Lambda_{n,n}$ admits a semisimple $(F,\fraci)$-isometry. 
\end{proposition}
\begin{proof}
If $J_- \cup J_+$ is one equivalence class then we are done by Lemma \ref{lem:if_J+=J-}. 
So we assume that $J_+$ and $J_-$ are distinct non-empty equivalence classes.   
Then $|F_{J_+^\circ}(-1)|$ and $|F_{J_-^\circ}(1)|$ are squares by 
Lemma \ref{lem:FJ_of_pm1_is_zero}. 
Thus, we have
\begin{equation}\label{eq*:J_pm_circ}
\begin{split}
(-1)^{\deg(F_{J_+^\circ})/2} F_{J_+^\circ}(1)F_{J_+^\circ}(-1)
&= e(F_{J_+^\circ})|F_{J_+^\circ}(1)||F_{J_+^\circ}(-1)|
= e(F_{J_+^\circ})|F_{J_+^\circ}(1)|,  \\
(-1)^{\deg(F_{J_-^\circ})/2} F_{J_-^\circ}(1)F_{J_-^\circ}(-1)
&= e(F_{J_-^\circ})|F_{J_-^\circ}(1)||F_{J_-^\circ}(-1)|
= e(F_{J_-^\circ})|F_{J_-^\circ}(-1)|  
\end{split}
\tag{$*$}  
\end{equation}
in $\bQ^\times/\bQ^{\times2}$. We also have the relation 
\[ e(F_{12}) 
= e(F_{J_+^\circ}) e(F_{J_-^\circ}) \times \prod_{H \neq J_+, J_-} e(F_H) \times e(F_2)
= e(F_{J_+^\circ}) e(F_{J_-^\circ}) 
\]
as in Equation \eqref{eq:e(F_12)}. 

\smallskip
\noindent
\textit{Case \textup{I}: $e(F_{12}) = 1$.}
We have $(e(F_{J_+^\circ}), e(F_{J_-^\circ})) = (1,1)$ or $(-1, -1)$. 
Furthermore $(i_+, i_-) \equiv (0,0)$ or $(2,2) \bmod 4$ by \eqref{eq:prolong}. 

\textit{Case \textup{I-(a)}: $(e(F_{J_+^\circ}), e(F_{J_-^\circ})) = (1,1)$ and 
$(i_+, i_-) \equiv (0,0) \bmod 4$.}
By applying Proposition \ref{prop:index_prolong} (ii) 
as $F = F_{J_+}$ and $(i_+, i_-) = (0, 0)$, 
we get $\fraci_{J_+}\in \Idx(n_{J_+}, n_{J_+}; F_{J_+})$ with 
$\fraci_{J_+}(X-1) = 0$. 
Similarly, there exists $\fraci_{J_-}\in \Idx(n_{J_-}, n_{J_-}; F_{J_-})$
with $\fraci_{J_-}(X+1) = 0$. 
Then we can obtain the desired $\fraci\in \Idx(n,n;F)$ as in Lemma \ref{lem:if_J+=J-}. 

\textit{Case \textup{I-(b)}: $(e(F_{J_+^\circ}), e(F_{J_-^\circ})) = (-1,-1)$ and 
$(i_+, i_-) \equiv (2,2) \bmod 4$.}
By applying Proposition \ref{prop:index_prolong} (ii) 
as $F = F_{J_+}$ and $(i_+, i_-) = (2, 0)$, 
we get $\fraci_{J_+}\in \Idx(n_{J_+}, n_{J_+}; F_{J_+})$ with 
$\fraci_{J_+}(X-1) = 2$. 
Similarly, there exists $\fraci_{J_-}\in \Idx(n_{J_-}, n_{J_-}; F_{J_-})$
with $\fraci_{J_-}(X+1) = 2$. So we are done as in Case I-(a). 

\textit{Case \textup{I-(c)}: $(e(F_{J_+^\circ}), e(F_{J_-^\circ})) = (1,1)$ and 
$(i_+, i_-) \equiv (2,2) \bmod 4$.}
We show that this case does not occur. Note that in $\bQ^\times/\bQ^{\times2}$ we have  
\[
(-1)^{\deg(F_{J_+^\circ})/2}F_{J_+^\circ}(1)F_{J_+^\circ}(-1) = |F_{J_+^\circ}(1)|, 
\quad
(-1)^{\deg(F_{J_-^\circ})/2}F_{J_-^\circ}(1)F_{J_-^\circ}(-1) = |F_{J_-^\circ}(-1)|
\]
by \eqref{eq*:J_pm_circ} and 
\[ \delta_+(F;i_+, i_-) = |F_{J_+^\circ}(1)|, 
\quad 
\delta_-(F;i_+, i_-) = |F_{J_-^\circ}(-1)|
\]
by Lemma \ref{lem:ECweakest}.    
%
%
These equations imply that 
\begin{alignat*}{4}
& X-1 \in \overline{I(F_{J_+^\circ}; \bQ_2)}' \qquad
&&\text{ if $|F_{J_+^\circ}(1)| = -1$ in $\bQ_2^\times/\bQ_2^{\times 2}$,} \\
& X-1 \in \overline{I(F_{J_-^\circ}; \bQ_2)}' 
&&\text{ if $|F_{J_-^\circ}(-1)| = -1$ in $\bQ_2^\times/\bQ_2^{\times 2}$,} \\
& X-1 \in \overline{I(X-1; \bQ_2)}' 
&&\text{ if $|F_{J_+^\circ}(1)|\neq -1$ in $\bQ_2^\times/\bQ_2^{\times 2}$,} \\
& X-1 \in \overline{I(X+1; \bQ_2)}' 
&&\text{ if $|F_{J_-^\circ}(-1)|\neq -1$ in $\bQ_2^\times/\bQ_2^{\times 2}$}
\end{alignat*}
by Proposition \ref{prop:nonSquare} (iii) and the definition of 
the set $\overline{I(f; \bQ_2)}'$. 
Hence, if $|F_{J_+^\circ}(1)| = -1$ and $|F_{J_-^\circ}(-1)| = -1$ in 
$\bQ_2^\times/\bQ_2^{\times 2}$ then 
$2\in \Pi_{i_+, i_-}^F(F_{J_+^\circ}, F_{J_-^\circ})$, 
and we would have $J_+ = J_-$; 
if $|F_{J_+^\circ}(1)| = -1$ and $|F_{J_-^\circ}(-1)| \neq -1$ in 
$\bQ_2^\times/\bQ_2^{\times 2}$ then 
$2\in \Pi_{i_+, i_-}^F(F_{J_+^\circ}, X+1)$, 
and we would have $J_+ = J_-$; 
if $|F_{J_+^\circ}(1)| \neq -1$ and $|F_{J_-^\circ}(-1)| = -1$ in 
$\bQ_2^\times/\bQ_2^{\times 2}$ then 
$2\in \Pi_{i_+, i_-}^F(X-1, F_{J_+^\circ})$, 
and we would have $J_+ = J_-$;  
if $|F_{J_+^\circ}(1)| \neq -1$ and $|F_{J_-^\circ}(-1)| \neq -1$ in 
$\bQ_2^\times/\bQ_2^{\times 2}$ then 
$2\in \Pi_{i_+, i_-}^F(X-1, X+1)$, 
and we would have $J_+ = J_-$.  
Therefore Case I-(c) does not occur. 

\textit{Case \textup{I-(d)}: $(e(F_{J_+^\circ}), e(F_{J_-^\circ})) = (-1,-1)$ and 
$(i_+, i_-) \equiv (0,0) \bmod 4$.}
This case does not occur for similar reasons to Case I-(c). 
So we are done in Case I. 

\smallskip
\noindent 
\textit{Case \textup{II}: $e(F_{12}) = -1$.}
We have $(e(F_{J_+^\circ}), e(F_{J_-^\circ})) = (1,-1)$ or $(-1, 1)$. 
Furthermore $(i_+, i_-) \equiv (2,0)$ or $(0,2) \bmod 4$ by \eqref{eq:prolong}. 
Hence there are $4$ cases; 
(a) $(e(F_{J_+^\circ}), e(F_{J_-^\circ})) = (1,-1), (i_+, i_-) \equiv (0,2)$; 
(b) $(e(F_{J_+^\circ}), e(F_{J_-^\circ})) = (-1,1), (i_+, i_-) \equiv (2,0)$; 
(c) $(e(F_{J_+^\circ}), e(F_{J_-^\circ})) = (1,-1), (i_+, i_-) \equiv (2,0)$; and
(d) $(e(F_{J_+^\circ}), e(F_{J_-^\circ})) = (-1,1), (i_+, i_-) \equiv (0,2)$. 
As in Case I, we obtain the desired $\fraci\in \Idx(n,n;F)$ in the cases 
(a) and (b), and it can be seen that the cases (c) and (d) do not occur. 
This completes the proof. 
\end{proof}

\begin{theorem}\label{th:chpl_on_index0}
Let $F\in \bZ[X]$ be a $*$-symmetric polynomial of even degree, and 
$i_+, i_-\in \bZ$ integers with \eqref{eq:prolong}. 
Suppose that $m_+ \neq 1$ and $m_- \neq 1$, where $m_\pm$ is the multiplicity 
of $X\mp1$ in $F$. 
Then there exists $\fraci \in \Idx(n,n;F)$ with 
$\fraci(X-1) \equiv i_+$ and $\fraci(X+1) \equiv i_-$ mod $4$
such that $\Lambda_{n,n}$ admits a semisimple $(F,\fraci)$-isometry. 
\end{theorem}
\begin{proof}
Suppose first that $m_+$ and $m_-$ are odd. Then 
$m_+, m_-\geq 3$ by the assumption $m_+, m_- \neq 1$. 
Thus $\Pi_{i_+,i_-}^F(X-1, X+1) = \Pi(X-1, X+1) \ni 2$. 
This implies that $J_+ = J_-$, and Lemma \ref{lem:if_J+=J-} leads to the assertion. 

Suppose then that $m_+$ and $m_-$ are even. 
If $m_+ = 0$ or $m_- = 0$ then $J_+ = \emptyset$ or $J_+ = \emptyset$, and  
we arrive at the assertion by Lemma \ref{lem:if_J+=J-}. 
Suppose that $m_+, m_- \geq 2$, and put $n_\pm' = (m_\pm - 2)/2$. 
Then $F$ can be written as 
\[ F(X) = (X-1)^{2n_+'}(X+1)^{2n_-'} \widetilde{F}(X),  \]
where $\widetilde{F}(X) := (X-1)^2(X+1)^2F_{12}(X)$. 
Let $\Lambda_\pm'$ and $\widetilde{\Lambda}$ denote the even unimodular lattices 
of signatures $(n_\pm', n_\pm')$ and $(\widetilde{n}, \widetilde{n})$, 
where $\widetilde{n} := n - n_+' - n_-'$.  
By Proposition \ref{prop:if_m+=m-=2}, there exists 
$\widetilde{\fraci} \in I(\widetilde{n},\widetilde{n};\widetilde{F})$ with 
$\widetilde{\fraci}(X-1) \equiv i_+$ and $\widetilde{\fraci}(X+1) \equiv i_-$ mod $4$
such that $\widetilde{\Lambda}$ admits a semisimple 
$(\widetilde{F}, \widetilde{\fraci})$-isometry $\widetilde{t}$. 
We now define an isometry $t$ of 
$\Lambda_{n,n} = \Lambda_+' \oplus \Lambda_-' \oplus \widetilde{\Lambda}$ 
by 
$t = \id_{\Lambda_+'} \oplus (-\id_{\Lambda_-'}) \oplus \widetilde{t}$. 
Then $t$ is a semisimple isometry of $\Lambda_{n,n}$ with 
characteristic polynomial $F$. Moreover, if $\fraci$ denotes the 
index of $t$ then 
$\fraci(X-1) \equiv i_+$ and $\fraci(X+1) \equiv i_-$ mod $4$
by the construction of $t$. This completes the proof. 
\end{proof}

This theorem is half of Theorem \ref{th:chpl_on_index0_intro}. 
On the question of whether the assumption $m_+, m_- \neq 1$ can be removed, 
the author has neither proof nor counterexample. 
However, we will see in the next subsection that
it is possible if the polynomial $F$ is a product of cyclotomic polynomials.

\subsection{Cyclotomic case} 
The aim of this subsection is to show the following theorem, 
which is the remaining half of Theorem \ref{th:chpl_on_index0_intro}. 

\begin{theorem}\label{th:chpl_on_index0_cyclo}
Let $F\in \bZ[X]$ be a product of cyclotomic polynomials. 
Assume that $\deg(F)$ is even, say $2n$. 
 Let $i_+, i_-\in \bZ$ be integers with \eqref{eq:prolong}. 
Then there exists $\fraci \in \Idx(n,n;F)$ with 
$\fraci(X-1) \equiv i_+$ and $\fraci(X+1) \equiv i_-$ mod $4$
such that $\Lambda_{n,n}$ admits a semisimple $(F,\fraci)$-isometry. 
\end{theorem}

In the following, we assume that $F$ is a product of cyclotomic polynomials
of degree $2n$. 
For a map $\fraci:I(F;\bR)\to \bZ$ and a factor $f\in I_1(F;\bQ)$ of $F$ in $\bQ[X]$, 
we write $\fraci(f)$ for the sum $\sum_{g\in I_1(f;\bR)}\fraci(g)$ 
(although it may be a slight abuse of notation).

\begin{lemma}\label{lem:4shift}
Let $r,s\in \bZ_{\geq 0}$ be non-negative integers with $r + s = 2n$. 
Suppose that $F$ satisfies the condition \eqref{eq:Sign}$_{r,s}$. 
Let $\fj \in \Idx(r,s;F)$ be an index map, and let $f\in I_1(F;\bQ)$ be a 
$+1$-symmetric irreducible factor of $F$ other than $X+1$. 
\begin{enumerate}
\item If $\fj(f) < \deg(f^{m_f})$ then $F$ satisfies the condition 
\eqref{eq:Sign}$_{r+2,s-2}$ and there exists $\fraci\in \Idx(r+2,s-2;F)$ such that 
\[ \fraci(h) = \begin{cases}
  \fj(h) + 4 & \text{if $h = f$}\\
  \fj(h) & \text{if $h \neq f$}
\end{cases}
\quad\text{for $h \in I(F;\bQ)$.}
\]
\item If $\fj(f) > - \deg(f^{m_f})$ then $F$ satisfies the condition 
\eqref{eq:Sign}$_{r-2,s+2}$ and there exists $\fraci\in \Idx(r-2,s+2;F)$ such that 
\[ \fraci(h) = \begin{cases}
  \fj(h) - 4 & \text{if $h = f$}\\
  \fj(h) & \text{if $h \neq f$}
\end{cases}
\quad\text{for $h \in I(F;\bQ)$.}
\]
\end{enumerate}
\end{lemma}
\begin{proof}
We remark that $f$ is a cyclotomic polynomial other than $X-1$ and $X+1$. This means 
that $f$ is of type $1$ over $\bR$. In other words, it can be written as 
$f = \prod_{g\in I_1(f;\bR)} g$. The multiplicity of each $g\in I_1(f;\bR)$ in $F$ 
is $m_f$. 
We prove the assertion (i). It is enough to show the existence of $\fraci$
by Corollary \ref{cor:conditionSign}. 
Suppose that $\fj(f) < \deg(f^{m_f})$. Then there exists 
$g_0\in I(f;\bR)$ such that $\fj(g_0) < \deg(g_0^{m_f})$. 
Noting that $\fj(g_0) \equiv \deg(g_0^{m_f}) \bmod 4$ by \eqref{eq:indexdata2}, 
we get $\fj(g_0) \geq \deg(g_0^{m_f}) - 4$. 
Thus, the map $\fraci:I(F;\bR)\to \bZ$ defined by 
\[ \fraci(g) = \begin{cases}
\fj(g) + 4 & \text{if $g = g_0$} \\
\fj(g) & \text{if $g \neq g_0$} 
\end{cases}
\quad \text{for $g\in I(F;\bR)$}
\]
satisfies the conditions \eqref{eq:indexdata1} and \eqref{eq:indexdata2}. Furthermore
\[ \sum_{g\in I(F;\bR)}\fraci(g) 
= r - s + 4
= (r+2) - (s-2).
\]
These mean that the map $\fraci$ belongs to $\Idx(r+2, s-2; F)$, and 
it is the desired index map. 
The assertion (ii) is obtained similarly. 
\end{proof}

In the situation of this lemma, we have 
\[ \eta_\infty(\fraci)(f) \neq \eta_\infty(\fj)(f) \quad \text{and} \quad
\eta_\infty(\fraci)(h) = \eta_\infty(\fj)(h) 
\text{ for all $h\in I(F;\bQ)\setminus\{f\}$}
\]
by \eqref{eq:eta_infty}. 

\begin{proposition}\label{prop:if_m+=m-=1}
Suppose that the multiplicities of $X-1$ and $X+1$ in $F$ are $1$. 
Then $\Lambda_{n,n}$ admits a semisimple isometry with characteristic polynomial $F$. 
\end{proposition}
\begin{proof}
We can write
\[ F(X) = (X-1)(X+1)f_1(X)^{m_1} f_2(X)^{m_2} \cdots f_l(X)^{m_l}, \]
where $f_1, \ldots, f_l$ are distinct cyclotomic polynomials and 
$m_1, \ldots, m_l$ are positive integers. 
Let $\fraci\in \Idx(n,n; F)$ be an index map, and take 
$\beta_\fraci = \{b_\fraci\}\cup\{b_p\}_{p\in \cV\setminus\{\infty\}} \in \cB_\fraci$. 
The following claim is essential.  

\smallskip
\noindent
\textit{Claim: If $\eta(\beta_\fraci)\neq \bm{0}$ then there exists $\fj\in\Idx(n,n;F)$ and 
$\beta_\fj\in \cB_\fj$ such that $\fj(X-1) = \fraci(X-1), \fj(X+1) = \fraci(X+1)$ and 
$\Supp(\eta(\beta_\fj))\subsetneq \Supp(\eta(\beta_\fraci))$.}
Note that 
$\eta(\beta_\fraci)(X-1) = 0$ and $\eta(\beta_\fraci)(X+1) = 0$ 
since each $b_v|_{M_v^\pm}$ is one dimensional, and 
$\eta(\beta_\fraci)\cdot \bm{1}_I = 0$ by Proposition \ref{prop:even1s}. 
Suppose that $\eta(\beta_\fraci)\neq \bm{0}$. 
Then, there exist $k,k'\in\{1, \ldots, l\}$ such that 
$\eta(\beta_\fraci)(f_k) = \eta(\beta_\fraci)(f_{k'}) = 1$. 
We assume $k = 1$ and $k' = 2$ without loss of generality. 
Note that we have 
$|\fraci(f_1) + \fraci(f_2)|\leq \deg(f_1^{m_1}) + \deg(f_2^{m_2})$
and the equality holds if and only if one of the following two conditions holds:
\begin{equation}\label{eq:maximal_idx}
\fraci(f_1) = \deg(f_1^{m_1}) \quad\text{and}\quad \fraci(f_2) = \deg(f_2^{m_2}), 
\end{equation}
\begin{equation}\label{eq:minimal_idx}
\fraci(f_1) = -\deg(f_1^{m_1}) \quad\text{and}\quad \fraci(f_2) = -\deg(f_2^{m_2}).
\end{equation}

\noindent
\textit{Case\textup{\Rnum{1}:}
$|\fraci(f_1) + \fraci(f_2)|< \deg(f_1^{m_1}) + \deg(f_2^{m_2})$.} 
We may assume that $\fraci(f_1) < \deg(f_1^{m_1})$. 
If $\fraci(f_2) > -\deg(f_2^{m_2})$ then 
it can be seen that there exists $\fj\in \Idx(n,n;F)$ such that 
\[ \fj(h) = \begin{cases}
\fraci(h) + 4 & \text{if $h = f_1$}\\
\fraci(h) - 4 & \text{if $h = f_2$}\\
\fraci(h) & \text{if $h \neq f_1, f_2$}
\end{cases}
\qquad (h\in I(F;\bQ))
\]
by using Lemma \ref{lem:4shift} twice. We now define 
a family $\beta_\fj$ of inner products by  
$\beta_\fj := \{b_\fj\}\cup \{b_p\}_{p\in \cV\setminus\{\infty\}}$, where 
$b_\fj$ is an inner product on $M_\infty$ such that $\idx_\alpha^{b_\fj} = \fj$. 
Then $\beta_\fj$ belongs to $\cB_\fj$. Moreover, we have 
\[ \begin{split}
\eta(\beta_\fj)(h) - \eta(\beta_\fraci)(h)
= \eta_\infty(\fj)(h) - \eta_\infty(\fraci)(h)
= \begin{cases}
1 &\text{if $h = f_1, f_2$} \\
0 &\text{if $h \neq f_1, f_2$}
\end{cases}
\end{split}
\]
by \eqref{eq:eta_infty}. Hence 
$\Supp(\eta(\beta_\fj)) = \Supp(\eta(\beta_\fraci))\setminus\{f_1, f_2\}
\subsetneq \Supp(\eta(\beta_\fraci))$ as required. 
If $\fraci(f_2) = -\deg(f_2^{m_2})$ then $\fraci(f_1) > -\deg(f_1^{m_1})$ 
by the assumption $|\fraci(f_1) + \fraci(f_2)|< \deg(f_1^{m_1}) + \deg(f_2^{m_2})$. 
In this case, there exists $\fj\in \Idx(n,n;F)$ such that 
\[ \fj(h) = \begin{cases}
\fraci(h) - 4 & \text{if $h = f_1$}\\
\fraci(h) + 4 & \text{if $h = f_2$}\\
\fraci(h) & \text{if $h \neq f_1, f_2$}
\end{cases}
\qquad (h\in I(F;\bQ))
\]
by Lemma \ref{lem:4shift}. Hence, as above, we obtain 
$\Supp(\eta(\beta_\fj)) \subsetneq \Supp(\eta(\beta_\fraci))$ for 
$\beta_\fj := \{b_\fj\}\cup \{b_p\}_p\in \cB_\fj$. 

\smallskip
\noindent
\textit{Case\textup{\Rnum{2}:}
$|\fraci(f_1) + \fraci(f_2)| = \deg(f_1^{m_1}) + \deg(f_2^{m_2})$ 
and \eqref{eq:maximal_idx} holds.}
In this case, we remark that if 
\begin{equation}\label{eq_star:cyclo}
\text{there exists $i \geq 3$ such that 
$\fraci(f_i)\leq -2$ and $\deg(f_i^{m_i})\geq 4$}
\tag{$\star$} 
\end{equation}
then we arrive at Claim. Indeed, if \eqref{eq_star:cyclo} holds then 
it can be shown that there exists $\fj\in \Idx(n,n;F)$ such that 
\[ \fj(h) = \begin{cases}
\fraci(h) - 4 & \text{if $h = f_1, f_2$}\\
\fraci(h) + 8 & \text{if $h = f_i$}\\
\fraci(h) & \text{if $h \neq f_1, f_2$}
\end{cases}
\qquad (h\in I(F;\bQ))
\]
by using Lemma \ref{lem:4shift} repeatedly. Therefore, as in Case\Rnum{1}, 
we obtain 
$\Supp(\eta(\beta_\fj))
= \Supp(\eta(\beta_\fraci))\setminus\{f_1, f_2\}
 \subsetneq \Supp(\eta(\beta_\fraci))$ for 
$\beta_\fj := \{b_\fj\}\cup \{b_p\}_p\in \cB_\fj$. 
We also remark that 
\begin{align}\label{eq:sum_over_kgeq3}
\textstyle\sum_{i=3}^l \fraci(f_i) 
&= - (\fraci(X-1) + \fraci(X+1)) - (\fraci(f_1) + \fraci(f_2)) \notag\\
&\leq 2 - (\deg(f_1^{m_1}) + \deg(f_2^{m_2}))  
\end{align}
by the assumption \eqref{eq:maximal_idx}. 

\smallskip
\textit{Case\textup{\Rnum{2}-(a):}
$\deg(f_1^{m_1}) + \deg(f_2^{m_2}) \geq 10$.}
It follows from equation \eqref{eq:sum_over_kgeq3} that $\sum_{i=3}^l \fraci(f_i) \leq -8$. 
Because the number of cyclotomic polynomials of degree $2$ is three (that is, 
$\Phi_3, \Phi_4$ and $\Phi_6$), the inequality $\sum_{i=3}^l \fraci(f) \leq -8$
shows that \eqref{eq_star:cyclo} holds. Thus we are done. 

\smallskip
\textit{Case\textup{\Rnum{2}-(b):}
$\deg(f_1^{m_1}) + \deg(f_2^{m_2}) = 8$.}
It follows from equation \eqref{eq:sum_over_kgeq3} that $\sum_{i=3}^l \fraci(f_i) \leq -6$. 
If $\deg(f_1^{m_1}) = 2$ or $\deg(f_1^{m_2}) = 2$ then 
the number of cyclotomic polynomials of degree $2$ which are contained in 
$\{f_i \mid i = 3,\ldots, l \}$ is at most two. 
Thus, the inequality $\sum_{i=3}^l \fraci(f_i) \leq -6$ leads to 
\eqref{eq_star:cyclo} as in Case\Rnum{2}-(a). 
Suppose that $\deg(f_1^{m_1}) = \deg(f_2^{m_2}) = 4$ and 
\eqref{eq_star:cyclo} does not hold. 
In this case, we have $l=5$, $\{f_3, f_4, f_5\} = \{ \Phi_3, \Phi_4, \Phi_6 \}$, 
$m_3 = m_4 = m_5 = 1$, and $\fraci(\Phi_3) = \fraci(\Phi_4) = \fraci(\Phi_6) = -2$.  
By Lemma \ref{lem:4shift}, we can take $\fj\in \Idx(n,n;F)$ such that 
\[ \fj(h) = \begin{cases}
\fraci(h) - 4 & \text{if $h = f_1, f_2$}\\
\fraci(h) + 4 & \text{if $h = \Phi_3, \Phi_6$}\\
\fraci(h) & \text{if $h \neq f_1, f_2, \Phi_3, \Phi_6$}
\end{cases}
\qquad (h\in I(F;\bQ))
\]
We define $\beta_\fj' := \{ b_\fj \} \cup \{b_p\}_p\in\cB_\fj$.  
Because $\Pi_\fj^F(\Phi_3, \Phi_6) = \Pi(\Phi_3, \Phi_6) = \{2\} \neq \emptyset$
by Theorem \ref{th:Pi_cyclos},  
there exists $\beta_\fj\in \cB_\fj$ such that 
$\eta(\beta_\fj) = \eta(\beta_\fj') + \bm{1}_{\{\Phi_3, \Phi_6\}}$
by Theorem \ref{th:imageofeta}. 
This is the desired family because 
$\Supp(\eta(\beta_\fj)) = \Supp(\eta(\beta_\fraci))\setminus\{f_1, f_2\}$.

\smallskip
\textit{Case\textup{\Rnum{2}-(c):}
$\deg(f_1^{m_1}) + \deg(f_2^{m_2}) = 6$.}
We assume that $\deg(f_1^{m_1}) = 2$ and $\deg(f_2^{m_2}) = 4$ without loss of
generality, and suppose that \eqref{eq_star:cyclo} does not hold. 
Then $l = 4$, $m_3 = m_4 = 1$, $\deg(f_3) = \deg(f_4) = 2$, and
$\fraci(f_3) = \fraci(f_4) = -2$. 
Note that $\{f_1, f_3, f_4\} = \{ \Phi_3, \Phi_4, \Phi_6 \}$. 
If $f_1 = \Phi_4$ then $\{f_3, f_4\} = \{ \Phi_3, \Phi_6 \}$, and 
we can obtain the desired index map $\fj\in \Idx(n,n;F)$ and family 
$\beta_\fj\in\cB_\fj$ as in Case\Rnum{2}-(b). 
Suppose that $f_1 \neq \Phi_4$. We may assume that $\{f_1 , f_3\} = \{\Phi_3, \Phi_6\}$. 
Since 
$\eta(\beta_\fraci)(X-1) = \eta(\beta_\fraci)(X+1) = 0$ and 
$\eta(\beta_\fraci)(f_1) = \eta(\beta_\fraci)(f_2) = 1$, 
one of the following two cases occurs by Proposition \ref{prop:even1s}:
(i) $\eta(\beta_\fraci)(f_3) = \eta(\beta_\fraci)(f_4) = 1$; or  
(ii) $\eta(\beta_\fraci)(f_3) = \eta(\beta_\fraci)(f_4) = 0$.  
Note that $\Pi^F_\fraci(f_1, f_3) = \Pi(\Phi_3, \Phi_6) = \{2\}$. 
Then, there exists a family $\beta_\fraci' = \{b_v'\}_{v\in \cV}\in \cB_\fraci$
such that $\eta(\beta_\fraci') = \eta(\beta_\fraci) + \bm1_{\{f_1, f_3\}}$ 
by Theorem \ref{th:imageofeta}.
In the case (i), this is the desired family because 
$\Supp(\eta(\beta_\fraci')) = \{ f_2, f_4 \} \subsetneq 
\{f_1, f_2, f_3, f_4\} = \Supp(\eta(\beta_\fraci))$. 
In the case (ii), there exists $\fj\in\Idx(n,n;F)$ such that 
\[ \fj(h) = \begin{cases}
\fraci(h) - 4 & \text{if $h = f_2$}\\
\fraci(h) + 4 & \text{if $h = f_4$}\\
\fraci(h) & \text{if $h \neq f_2, f_4$}
\end{cases}
\qquad (h\in I(F;\bQ))
\]
by Lemma \ref{lem:4shift}. We define $\beta_j := \{ b_\fj \} \cup \{b_p'\}_p\in\cB_\fj$.
Then $\Supp(\eta(\beta_\fj)) = \emptyset$ since 
$\Supp(\eta(\beta_\fraci')) = \{f_2, f_4\}$. 
Hence we are done. 

\smallskip
\textit{Case\textup{\Rnum{2}-(d):}
$\deg(f_1^{m_1}) + \deg(f_2^{m_2}) = 4$.}
This case is similar to Case\Rnum{2}-(c). 
We have $m_1 = m_2 = 1$ and $\deg(f_1) = \deg(f_2) = 2$. 
Suppose that \eqref{eq_star:cyclo} does not hold. 
Then $l = 3$, $m_3 = 1$, $\deg(f_3) = 2$, and $\fraci(f_3) = -2$. 
Thus $\{f_1, f_2, f_3\} = \{ \Phi_3, \Phi_4, \Phi_6 \}$, and moreover 
we have $\eta(\beta)(f_3) = 0$ by Proposition \ref{prop:even1s}. 
If $\{ f_1, f_2 \} = \{ \Phi_3, \Phi_6 \}$ then 
a family $\beta_\fraci'\in \cB_\fraci$ with 
$\eta(\beta_\fraci') = \eta(\beta_\fraci) + \bm1_{\{f_1, f_2\}}$
is the desired one. 
Suppose that $\{ f_1, f_2 \} \neq \{ \Phi_3, \Phi_6 \}$. 
We may assume $\{ f_1, f_3 \} = \{ \Phi_3, \Phi_6 \}$. 
Let $\beta_\fraci' = \{b_v'\}_{v\in \cV}\in \cB_\fraci$ be a family with 
$\eta(\beta_\fraci') = \eta(\beta_\fraci) + \bm1_{\{f_1, f_3\}}$, and 
define $\beta_\fj := \{ b_\fj \} \cup \{b_p'\}_p\in\cB_\fj$, where 
$\fj \in \Idx(n,n;F)$ is an index map satisfying 
\[ \fj(h) = \begin{cases}
\fraci(h) - 4 & \text{if $h = f_2$}\\
\fraci(h) + 4 & \text{if $h = f_3$}\\
\fraci(h) & \text{if $h \neq f_2, f_3$}
\end{cases}
\qquad (h\in I(F;\bQ)).
\]
Then $\Supp\eta(\beta_\fj) = \emptyset$, and we are done. 

\smallskip
\noindent
\textit{Case\textup{\Rnum{3}:}
$|\fraci(f_1) + \fraci(f_2)| = \deg(f_1^{m_1}) + \deg(f_2^{m_2})$ 
and \eqref{eq:minimal_idx} holds.}
In this case, Claim can be shown as in Case\Rnum{2}. 
The proof of Claim is complete. 

\smallskip

By applying Claim repeatedly, we obtain $\fj\in\Idx(n,n;F)$ and 
$\beta_\fj\in \cB_\fj$ such that $\fj(X-1) = \fraci(X-1), \fj(X+1) = \fraci(X+1)$ and 
$\eta(\beta_\fj) = \bm{0}$.
This means that $\Lambda_{n,n}$ admits a semisimple $(F,\fj)$-isometry
by Theorem \ref{th:obstruction}. The proof is complete. 
\end{proof}

\begin{proof}[Proof of Theorem \textup{\ref{th:chpl_on_index0_cyclo}}.]
If $m_+$ and $m_-$ are even then the theorem follows from Theorem 
\ref{th:chpl_on_index0}. 
Suppose that $m_+$ and $m_-$ are odd, and put $n_\pm' = (m_\pm-1)/2$. 
Then $F$ can be written as 
$F(X) = (X-1)^{2n_+'}(X+1)^{2n_-'}\widetilde{F}(X)$, 
where $\widetilde{F}(X) = (X-1)(X+1)F_{12}(X)$. 
Let $\Lambda_\pm'$ and $\widetilde\Lambda$ denote the even unimodular lattices 
of signatures $(n_\pm', n_\pm')$ and $(\widetilde{n}, \widetilde{n})$, 
where $\widetilde{n} := n - n_+ - n_-$.  
By Proposition \ref{prop:if_m+=m-=1}, there exists a semisimple isometry 
$\widetilde{t}$ of $\widetilde\Lambda$ with characteristic polynomial $F$. 
Let $\widetilde{b}$ denote the inner product of the lattice $\widetilde\Lambda$. 
Note that $(\widetilde\Lambda, -\widetilde{b})$ is also an even unimodular lattice of 
signature $(\widetilde{n}, \widetilde{n})$, and 
$\widetilde{t}$ is a semisimple $(F, - \widetilde{\fraci})$-isometry
of $(\widetilde\Lambda, -\widetilde{b})$, where 
$\widetilde\fraci := \idx_{\widetilde{t}}^{\widetilde{b}}$. 

Suppose that $e(F_{12})=1$. Since 
$\widetilde\fraci(X-1) \equiv \widetilde\fraci(X+1) \equiv 1 \bmod 2$ 
by \eqref{eq:idx1} and 
$\widetilde\fraci(X-1) + \widetilde\fraci(X+1) \equiv 0 \bmod 4$ 
by Proposition \ref{prop:index_prolong} (i), we have 
$(\widetilde\fraci(X-1), \widetilde\fraci(X+1)) \equiv (1,-1)$ or $(-1,1) \bmod 4$. 
Similarly, we have $(i_+, i_-) \equiv (1,-1)$ or $(-1,1) \bmod 4$ 
by the assumption \eqref{eq:prolong}. 
Hence, we assume without loss of generality that 
$(\widetilde\fraci(X-1), \widetilde\fraci(X+1)) \equiv (i_+, i_-) \bmod 4$
by replacing $\widetilde{b}$ by $-\widetilde{b}$ if necessary. 
We now define an isometry $t$ of 
$\Lambda_{n,n} = \Lambda_+' \oplus \Lambda_-' \oplus \widetilde\Lambda$ by 
$t := \id_{\Lambda_+'} \oplus (-\id_{\Lambda_-'}) \oplus \widetilde{t}$. 
Then $t$ is a semisimple isometry of $\Lambda_{n,n}$ with characteristic polynomial $F$, 
and its index is the desired index map by construction. 
In the case $e(F_{12})=-1$, we can obtain the desired index map similarly.  
The proof is complete. 
\end{proof}

\begin{proof}[Proof of Theorem \textup{\ref{th:chpl_on_index0_intro}}.]
This is the union of Theorems \ref{th:chpl_on_index0} and \ref{th:chpl_on_index0_cyclo}. 
\end{proof}

\section{Dynamical degrees of K3 surface automorphisms}\label{sec:DDofK3SA}
A \textit{K3 surface} is a compact complex surface $\Sigma$ such that its canonical 
bundle is trivial and $\dim_\bC(H^{0,1}(\Sigma)) = 0$. 
For any K3 surface $\Sigma$, the second cohomology group $H^2(\Sigma,\bZ)$
is a free $\bZ$-module of rank $22$, and the intersection form makes $H^2(\Sigma,\bZ)$
an even unimodular lattice of signature $(3,19)$. 
We refer to an even unimodular lattice of signature $(3,19)$ as a \textit{K3 lattice}. 
Let $\phi$ be an automorphism of a K3 surface $\Sigma$. 
Then, the induced homomorphism $\phi^*:H^2(\Sigma, \bZ) \to H^2(\Sigma, \bZ)$
is an isometry of the lattice $H^2(\Sigma, \bZ)$. 
Let $d(\phi)$ denote the spectral radius of 
$\phi^*:H^2(\Sigma, \bC) \to H^2(\Sigma, \bC)$:
\[ d(\phi) := \max\{ |\mu| \mid \text{$\mu\in \bC$ is an eigenvalue of
$\phi^*:H^2(X,\bC) \to H^2(X,\bC)$} \}. 
\] 
We refer to this value $d(\phi)$ as the \textit{dynamical degree} of $\phi$. 
It is known that the entropy of $\phi$ is given by $\log d(\phi)$. 
It is also known that if an automorphism $\phi$ of a K3 surface $\Sigma$ has 
dynamical degree greater than $1$ then it is a \textit{Salem number}, 
that is, a real algebraic number $\lambda > 1$
such that it is conjugate to $\lambda^{-1}$ and all of its conjugates other than 
$\lambda$ and $\lambda^{-1}$ have absolute value $1$. 
More strongly, in this case, the induced homomorphism 
$\phi^*:H^2(\Sigma, \bC)\to H^2(\Sigma, \bC)$ is semisimple, and 
its characteristic polynomial is of the form $F = SC$, 
where $S$ is a \textit{Salem polynomial}, i.e., the minimal polynomial of a Salem number,  
and $C$ is a product of cyclotomic polynomials. 
We refer to \cite[\S 20]{Ba22} and references therein for these facts. 

Our concern is to know which Salem number can be realized as 
the dynamical degree of an automorphism of a K3 surface.  

\begin{definition}
We use the following terminology. 
\begin{enumerate}
\item We refer to a polynomial $F(X)\in \bZ[X]$ of degree $22$ as a 
\textit{complemented Salem polynomial} if it can be expressed as $F=SC$, 
where $S$ is a Salem polynomial and $C$ is a product of cyclotomic polynomials.
In this case, $S$ is called the \textit{Salem factor} of $F$. 
\item A Salem number $\lambda$ is 
\textit{projectively} (resp.~\textit{nonprojectively}) \textit{realizable}
if there exists an automorphism of a projective (resp. nonprojective)
K3 surface with dynamical degree $\lambda$.
\end{enumerate}
\end{definition}

Let $F = SC$ be a complemented Salem polynomial with Salem factor $S$. 
Note that $F$ satisfies the condition \eqref{eq:Sign}$_{3, 19}$. 
For a root $\delta$ of $S$ with $|\delta| = 1$, we define the index map 
$\fraci_\delta \in \Idx(3,19;F)$ by 
\[ \fraci_\delta = \begin{cases}
2  & \text{if $f(X)= X^2 - (\delta + \delta^{-1})X + 1$} \\
-\deg(f^{m_f}) & \text{if $f(X)\neq X^2 - (\delta + \delta^{-1})X + 1$}
\end{cases}
\quad\text{for $f\in I(F;\bR)$.}
\]
We remark that any Salem number has even degree, and 
any nonprojectively realizable Salem number has degree at least $4$
(this fact follows from \cite[Proposition 20.6]{Ba22}). 

\begin{proposition}\label{prop:nonproj_realizable}
Let $\lambda$ be a Salem number with $4\leq \deg\lambda \leq 22$, and 
$S$ its minimal polynomial. The following are equivalent:
\begin{enumerate}
\item $\lambda$ is nonprojectively realizable.
\item There exists a conjugate $\delta$ of $\lambda$ with $|\delta| = 1$ and 
a complemented Salem polynomial $F$ with Salem factor $S$ such that
a K3 lattice admits a semisimple $(F, \fraci_\delta)$-isometry. 
\item There exists a conjugate $\delta$ of $\lambda$ with $|\delta| = 1$ and 
a complemented Salem polynomial $F$ with Salem factor $S$ such that
it satisfies \eqref{eq:Square} and the obstruction map for $(F, \fraci_\delta)$ vanishes. 
\item For any conjugate $\delta$ of $\lambda$ with $|\delta| = 1$, there exists 
a complemented Salem polynomial $F$ with Salem factor $S$ such that
it satisfies \eqref{eq:Square} and the obstruction map for $(F, \fraci_\delta)$ vanishes. 
\item For any conjugate $\delta$ of $\lambda$ with $|\delta| = 1$, there exists 
a complemented Salem polynomial $F$ with Salem factor $S$ such that
a K3 lattice admits a semisimple $(F, \fraci_\delta)$-isometry. 
\end{enumerate}
\end{proposition}
\begin{proof}
See \cite[Proposition 7.4]{Ta23} for the equivalence (i) $\Leftrightarrow$ (ii).  
The implications (ii) $\Rightarrow$ (iii) and (iv) $\Rightarrow$ (v) 
follow from Theorem \ref{th:obstruction}. 
Furthermore (v) $\Rightarrow$ (ii) is obvious. So it remains to prove 
(iii) $\Rightarrow$ (iv). Let $\delta'$ be a conjugate of $\lambda$ with $|\delta'| = 1$, 
and suppose that there exists a complemented Salem polynomial $F$ with Salem factor $S$ 
such that it satisfies \eqref{eq:Square} and the obstruction map for 
$(F, \fraci_{\delta'})$ vanishes. Let $\delta$ be another conjugate of $\lambda$ 
with $|\delta| = 1$. Then, Theorem \ref{th:comparison} shows that 
the obstruction map for $(F, \fraci_\delta)$ is also zero 
because $\eta_\infty(\fraci_\delta) = \eta_\infty(\fraci_{\delta'})$.  
This completes the proof. 
\end{proof}

In the following, $\lambda$ is a Salem number of degree $d$ with $4\leq d \leq 22$, and 
$S$ is its minimal polynomial. Furthermore, we fix a conjugate $\delta$ of $\lambda$
with $|\delta| = 1$. 

\begin{theorem}\label{th:dleq18_notSq}
Suppose that $d\leq 18$. 
If $S$ does not satisfy the condition \eqref{eq:Square} then 
$\lambda$ is nonprojectively realizable. 
\end{theorem}
\begin{proof}
Suppose first that $|S(1)|$ is not a square. 
Put $F(X) = (X-1)^{21-d}(X+1)S(X)$. Then $F$ is a complemented Salem polynomial 
with the condition \eqref{eq:Square}. Since the multiplicity of $X-1$ in $F$, $21-d$, 
is greater than $3$, Proposition \ref{prop:nonSquare} (i) implies that 
$\Pi_{\fraci_\delta}^F(S, X-1)$ is not empty.  
Thus, the obstruction group for $(F,\fraci_\delta)$ is generated by 
$\bm1_{\{S, X-1\}}$ and $\bm1_{\{X+1\}}$. 
Now, let $\{b_v\}_{v\in\cV}$ be a family in $\cB_{\fraci_\delta}$, 
where $\cB_{\fraci_\delta}$ is defined in Notation \ref{nt:P1P2P3}. 
Since the multiplicity of $X+1$ in $F$ is $1$, we have 
\[ \eta(\{b_v\}_{v})\cdot \bm1_{\{X+1\}} = \eta(\{b_v\}_{v})(X+1) 
= \sum_{v\in \cV} \hw_v(b_v|_{M_v^-}) = 0
\]
where we use notation in \S\ref{sec:LGPO}. Moreover
\[ \eta(\{b_v\}_{v})\cdot \bm1_{\{S, X-1\}}  
= \eta(\{b_v\}_{v})\cdot \bm1_{\{S, X-1\}} + \eta(\{b_v\}_{v})\cdot \bm1_{\{X+1\}}
= \eta(\{b_v\}_{v})\cdot \bm1_{\{S, X-1, X+1\}}
= 0, 
\]
where the last equation is by Proposition \ref{prop:even1s}. 
These mean that the obstruction map for $(F, \fraci_\delta)$ vanishes. 
Hence $\lambda$ is nonprojectively realizable by Proposition 
\ref{prop:nonproj_realizable}. 
Similarly, if $|S(-1)|$ is not a square then it can be checked that
$\lambda$ is nonprojectively realizable by putting $F(X) = (X-1)(X+1)^{21-d}S(X)$. 

Suppose then that $|S(1)|$ and $|S(-1)|$ are squares but 
$(-1)^{d/2}S(1)S(-1)$ is not a square. In this case, we put 
$F(X) = (X-1)^{22-d}S(X)$. This is a complemented Salem polynomial 
with the condition \eqref{eq:Square}. 
Since $(-1)^{d/2}S(1)S(-1) = -1$ in $\bQ^\times/\bQ^{\times2}$ and 
$22-d\geq 3$, Proposition \ref{prop:nonSquare} (iii) implies that 
$\Pi_{\fraci_\delta}^F(S, X-1)$ contains $2$ and is not empty. This means that 
the equivalence relation on $I(F;\bQ) = \{X-1, S\}$ defined by 
$(F, \fraci_\delta)$ is weakest. So $\lambda$ is nonprojectively realizable by 
Theorem \ref{th:weakestER} and Proposition \ref{prop:nonproj_realizable}. 
Therefore, if $S$ does not satisfy the condition \eqref{eq:Square} then 
$\lambda$ is nonprojectively realizable. 
\end{proof}

\begin{remark}
\begin{enumerate}
\item In the proof of Theorem \ref{th:dleq18_notSq}, the complemented Salem polynomial $F$
can be chosen in a different way. For example, put 
$F(X) = (X-1)^{(22-d)/2}(X-1)^{(22-d)/2}S(X)$. 
If $d \leq 16$ then one can show that the equivalence relation on 
$I(F;\bQ) = \{X-1, X+1, S\}$ defined by $(F, \fraci_\delta)$ is weakest, 
and conclude that $\lambda$ is nonprojectively realizable. 
\item One can show that any Salem polynomial whose degree is a multiple of $4$ 
does not satisfy \eqref{eq:Square}. So, it follows from Theorem \ref{th:dleq18_notSq}
that if $d = 4,8,12$, or $16$ then $\lambda$ is nonprojectively realizable. 
\end{enumerate}
\end{remark}

For $d = 10$ or $18$, let $\widetilde\cC_d$ denote the set of integers $l\geq 1$ such 
that $\varphi(l)<22-d$, or $\varphi(l)=22-d$ and $\Phi_{l}$ satisfies \eqref{eq:Square}. 
Integers $l$ with $\varphi(l) \leq 12$ are listed in Table \ref{table:fibers_of_varphi}. 

\begin{table}
\caption{Integers $l$ with $\varphi(l) \leq 12$}
\label{table:fibers_of_varphi}
\centering
\begin{tabular}{ll}\hline
$\varphi(l)$ & $l$ \\\hline\hline
$1$ & $1, 2$ \\
$2$ & $3,4,6$ \\
$4$ & $5,8,10,12$ \\
$6$ & $7,9,14,18$ \\
$8$ & $15,16,20,24,30$ \\
$10$ & $11,22$ \\
$12$ & $21,28,36,42$ \\\hline
\end{tabular}
\end{table}

\begin{lemma}\label{lem:cC_d}
Suppose that $d = 10$ or $18$, and let $F=SC$ be a complemented Salem polynomial
with Salem factor $S$.  
Suppose that $F$ and $S$ satisfies the condition \eqref{eq:Square}. Then 
$C$ is also satisfies \eqref{eq:Square}, and any irreducible factor of $C$ 
can be written as $\Phi_l$ for some $l\in \widetilde\cC_d$.
\end{lemma}
\begin{proof}
It follows from Lemma \ref{lem:properties_for_Square} (i) that  
$C$ is satisfies \eqref{eq:Square}. 
Let $\Phi_l$ be an irreducible factor of $C$, where $l\in \bZ_{>0}$. Then
$ \varphi(l) = \deg(\Phi_l) \leq \deg(C) = 22 - d. $ 
Moreover, if $\varphi(l) = 22-d$ then $\Phi_l= C$, and it satisfies 
\eqref{eq:Square}. This completes the proof. 
\end{proof}

Put $\cC_{10} = \widetilde\cC_{10}\setminus\{20\}$ and $\cC_{18}= \widetilde\cC_{18}$.  
These sets can be expressed explicitly as 
\[ \begin{split}
\cC_{10} &= \{1,2,3,4,5,6,7,8,9,10,11,12,14,15,16,18,21,22,24,28,30,36,42\}, \\
\cC_{18} &= \{1,2,3,4,6,12\},  
\end{split}
\]
as defined in Theorem \ref{th:NPR}. 

\begin{proposition}\label{prop:d=10or18_ness}
Suppose that $d = 10$ or $18$. If $\lambda$ is nonprojectively realizable
then there exists $l\in \cC_d$ such that $\Pi(S, \Phi_l)\neq \emptyset$. 
\end{proposition}
\begin{proof}
If $S$ does not satisfy \eqref{eq:Square}, then $\Pi(S, \Phi_1)\neq \emptyset$ or 
$\Pi(S, \Phi_2)\neq \emptyset$ by Proposition \ref{prop:nonSquare}, and we are done. 
So we assume that $S$ satisfies \eqref{eq:Square}. 
Suppose that $\lambda$ is nonprojectively realizable. Then, by Proposition 
\ref{prop:nonproj_realizable}, there exists a complemented Salem polynomial $F=SC$
with the condition \eqref{eq:Square} 
such that the obstruction map $\ob_{\fraci_\delta}$ for $(F, \fraci_\delta)$ is zero. 
We remark that $C$ satisfies \eqref{eq:Square}, and 
any irreducible factor of $C$ can be written as $\Phi_l$ for some 
$l\in \widetilde\cC_d$ by Lemma \ref{lem:cC_d}. 

Let $\Lambda_S$ and $\Lambda_C$ be even unimodular lattices of signature 
$(d/2, d/2)$  and $((22-d)/2, (22-d)/2)$ respectively. Let 
$\fj_S\in \Idx(d/2, d/2; S)$ be an index map. Then $\Lambda_S$ admits a semisimple
$(S, \fj_S)$-isometry $t_S$ by Theorem \ref{th:weakestER} since $S$ is irreducible. 
On the other hand, Theorem \ref{th:chpl_on_index0_cyclo} shows that 
there exists $\fj_C\in \Idx((22-d)/2, (22-d)/2; F)$ with 
\begin{equation}\label{eq*:congruence}
\text{$\fj_C(X-1)\equiv \fraci_\delta(X-1)$ 
and $\fj_C(X+1)\equiv \fraci_\delta(X+1)$ mod $4$} \tag{$*$}
\end{equation}
such that $\Lambda_C$ admits a semisimple $(C, \fj_C)$-isometry $t_C$. 
Then $t := t_S \oplus t_C$ is a semisimple $(F,\fj)$-isometry on the even 
unimodular lattice $\Lambda_S\oplus \Lambda_C$, where 
$\fj := \fj_S \oplus \fj_C \in \Idx(11,11;F)$. 
In particular, the obstruction map $\ob_\fj$ for $(F, \fj)$ vanishes. 
Note that the equivalence relation on $I(F;\bQ)$ defined by $(F, \fj)$ is the 
same as the one defined by $(F, \fraci_\delta)$ because of the relation 
\eqref{eq*:congruence}. Let $\sim$ denote the equivalence relation on $I(F;\bQ)$, 
and $\Omega$ the obstruction group. 

Now, suppose that $\Pi(S, \Phi_l)$ were empty for all $l \in \widetilde\cC_d$. 
Then $\{S\} \subset I(F;\bQ)$ forms an equivalence class with respect to $\sim$, 
since any irreducible factor of $C$ can be written as $\Phi_l$ for some 
$l\in \widetilde\cC_d$. 
This would mean that $\bm1_{\{S\}}$ belongs to $\Omega$. On the other hand, 
a calculation yields 
\[ (\eta_\infty(\fraci_\delta) - \eta_\infty(\fj)) \cdot \bm1_{\{S\}}
= 1-0 = 1 \quad \text{(in $\bZ/2\bZ$)}.  
\] 
This contradicts Theorem \ref{th:comparison} since $\ob_{\fraci_\delta}$ and 
$\ob_\fj$ are zero. Therefore,  
there exists $l \in \widetilde\cC_d$ such that $\Pi(S, \Phi_l)\neq \emptyset$. 
This complete the proof of the case $d = 18$ since $\cC_{18} = \widetilde\cC_{18}$. 

Suppose that $d = 10$. It remains to prove that the following case does not occur: 
$20$ is the only integer in $\widetilde\cC_{10}$ such that 
$\Pi(S, \Phi_{20})\neq \emptyset$. 
Suppose that this case occurred. Then $C$ must be divisible by $\Phi_{20}$
by the same reason as above, and $F$ can be expressed as 
$F = S\Phi_{20}C'$, where $C'$ is the remaining factor of $C$.  
Note that $\Phi_{20}(X) = X^8 - X^6 + X^4 - X^2 + 1$ satisfies \eqref{eq:Square}, 
and we may assume that $\fj(\Phi_{20}) = 0$. 
Because $S\Phi_{20}$ has degree $18$ and satisfies \eqref{eq:Square}, 
any factor of $C'$ can be written as $\Phi_l$ for some $l \in \cC_{18}$ by the 
same reason as Lemma \ref{lem:cC_d}. 
By using Theorem \ref{th:Pi_cyclos}, it can be checked that 
$\Pi(\Phi_{20}, \Phi_l) = \emptyset$ for all $l \in \cC_{18}$. This would mean that 
$\{S, \Phi_{20}\} \subset I(F;\bQ)$ forms an equivalence class with respect to $\sim$, 
and $\bm1_{\{S, \Phi_{20}\}} \in \Omega$. 
On the other hand, we have 
\[ \begin{split}
(\eta_\infty(\fraci_\delta) - \eta_\infty(\fj)) \cdot \bm1_{\{S, \Phi_{20}\}} 
&= \eta_\infty(\fraci_\delta)(S) - \eta_\infty(\fj)(S)
 + \eta_\infty(\fraci_\delta)(\Phi_{20}) - \eta_\infty(\fj)(\Phi_{20}) \\
&= 1 - 0 + 0 - 0 \\
&= 1. 
\end{split} \] 
However, this contradicts Theorem \ref{th:comparison}. 
Hence, there exists $l\in \cC_{10} = \widetilde\cC_{10}\setminus\{20\}$ such that 
$\Pi(S, \Phi_l)\neq \emptyset$. This completes the proof. 
\end{proof}

\begin{proposition}\label{prop:d=10or18_suff}
Suppose that $d = 10$ or $18$. 
If there exists $l\in \cC_{d}$ such that $\Pi(S, \Phi_l)\neq \emptyset$ then 
$\lambda$ is nonprojectively realizable. 
\end{proposition}
\begin{proof}
If $S$ does not satisfy \eqref{eq:Square} then $\lambda$ is nonprojectively realizable
by Theorem \ref{th:dleq18_notSq}. So we assume that $S$ satisfies \eqref{eq:Square}. 
Suppose that there exists $l\in \cC_d$ such that $\Pi(S, \Phi_l)\neq \emptyset$. 

\smallskip
\noindent
\textit{Case $d = 18$.} We define a complemented Salem polynomial $F$ as 
\[ F(X) = \begin{cases}
S(X)\Phi_l(X)^4 & \text{if $l = 1,2$} \\
S(X)\Phi_l(X)^2 & \text{if $l = 3,4,6$} \\
S(X)\Phi_l(X) & \text{if $l = 12$}. 
\end{cases}
\]
Then $F$ satisfies \eqref{eq:Square}, and 
$\Pi_{\fraci_\delta}^F(S, \Phi_l) = \Pi(S, \Phi_l) \neq \emptyset$. Thus, 
the equivalence relation on $I(F;\bQ) = \{S, \Phi_l\}$ defined by $(F, \fraci_\delta)$ 
is weakest, and hence $\lambda$ is nonprojectively realizable by 
Theorem \ref{th:weakestER} and Proposition \ref{prop:nonproj_realizable}.

\smallskip
\noindent
\textit{Case $d = 10$.} We define a complemented Salem polynomial $F$ as 
\[ F(X) = \begin{cases}
S(X)\Phi_l(X)^{12/\varphi(l)} & \text{if $\varphi(l) = 1,2,6,12$} \\
S(X)\Phi_l(X)(X-1)^4(X+1)^4 & \text{if $\phi(l) = 4$ but $l\neq 12$, i.e., $l=5,8,10$} \\
S(X)\Phi_{12}(X)^3& \text{if $l= 12$} \\
S(X)\Phi_l(X)\Phi_3(X)^2 & \text{if $l = 15, 24$} \\ 
S(X)\Phi_{16}(X)(X-1)^4 & \text{if $l = 16$} \\ 
S(X)\Phi_{30}(X)\Phi_6(X) & \text{if $l = 30$} \\ 
S(X)\Phi_{11}(X)(X-1)^2 & \text{if $l = 11$} \\ 
S(X)\Phi_{22}(X)(X+1)^2 & \text{if $l = 22$}. 
\end{cases}.
\]
Then $F$ satisfies \eqref{eq:Square}, and it can be checked that 
the equivalence relation on $I(F;\bQ)$ defined by $(F, \fraci_\delta)$ 
is weakest by using Theorem \ref{th:Pi_cyclos}. 
Hence, we conclude that $\lambda$ is nonprojectively realizable by 
Theorem \ref{th:weakestER} and Proposition \ref{prop:nonproj_realizable}. 
\end{proof}

\begin{proof}[Proof of Theorem \textup{\ref{th:NPR}}]
This is a consequence of Propositions 
\ref{prop:d=10or18_ness} and \ref{prop:d=10or18_suff}. 
\end{proof}

As mentioned in Introduction, 
Bayer-Fluckiger gives an example of a Salem number of degree $18$ that is not 
realizable as the dynamical degree of an automorphism of a K3 surface, 
projective or not, see \cite[\S 26]{Ba22}.  
According to Theorem \ref{th:NPR}, it seems that any Salem number of degree $10$ 
is nonprojectively realizable, but the author has no proof. 

\begin{remark}\label{otherdegrees}
For the other degrees, the following results are known.
Let $\lambda$ be a Salem number of degree $d$ with $4\leq d \leq 22$, and 
$S$ its minimal polynomial. 
\begin{enumerate}
\item Suppose that $d = 22$. Then $\lambda$ is nonprojectively realizable
if and only if $|S(1)|$ and $|S(-1)|$ are squares. 
\item Suppose that $d = 4,6,8,12,14,16$, or $20$. Then $\lambda$ is nonprojectively 
realizable. 
\end{enumerate}
The proof of (i) was given by Bayer-Fluckiger and Taelman in \cite{BT20} for the 
first time. The assertion (ii) was proved by Bayer-Fluckiger \cite{Ba21, Ba22} 
for $d = 4,6,8,12,14,16$, and by the author \cite{Ta23} for $d = 20$. 

As for projective realizability, there is a (computer-aided) criterion, 
see \cite{Mc16} and \cite{BG18}. 
However, a characterization of projectively realizable Salem numbers has not yet 
been obtained. There is a result of another type by S. Brandhorst, see \cite{Br20}. 
\end{remark}

\end{document}